\documentclass{amsart}
\usepackage[utf8]{inputenc}

\usepackage{graphics}
\usepackage{thmtools}
\usepackage{wasysym}
\usepackage[T1]{fontenc}    % use 8-bit T1 fonts

\usepackage{amsthm}
\usepackage{amsbsy,amsmath,amssymb,amscd,amsfonts}
\usepackage[pagebackref=false]{hyperref}
\usepackage[nameinlink,capitalize,noabbrev]{cleveref}

\usepackage{graphicx,float,latexsym,color}
\usepackage[font={scriptsize,it}]{caption}
\usepackage{subcaption}

\usepackage[export]{adjustbox}
\usepackage[font={scriptsize,it}]{caption}
\usepackage{makecell}

\usepackage[dvipsnames]{xcolor}

\newtheorem{theorem}{Theorem}
\newtheorem*{theorem*}{Theorem}

\newtheorem{proposition}{Proposition}
\newtheorem{conjecture}{Conjecture}
\newtheorem{corollary}{Corollary}
\newtheorem{lemma}{Lemma}
\theoremstyle{remark}
\newtheorem{remark}{Remark}
\theoremstyle{definition}
\newtheorem{definition}{Definition}

\hypersetup{
    pdftoolbar=true,        % show Acrobat’s toolbar?
    pdfmenubar=true,        % show Acrobat’s menu?
    pdffitwindow=false,     % window fit to page when opened
    pdfstartview={FitH},    % fits the width of 
    colorlinks=true,       % false: boxed links; true: colored links
    linkcolor=OliveGreen,          % color of internal links (change box color with linkbordercolor)
    citecolor=blue,        % color of links to bibliography
    filecolor=black,      % color of file links
    urlcolor=red           % color of external links
}

\usepackage{comment}
\usepackage{microtype}
\usepackage{footnote}

\newcommand{\E}{\mathcal{E}}
\newcommand{\T}{\mathcal{T}}
\newcommand{\C}{\mathcal{C}}

\newcommand{\D}{\mathcal{D}}

\usepackage{paralist}

\title{Loci of Poncelet Triangles\\in the General Closure Case}

\author{Ronaldo Garcia} 
\thanks{R. Garcia, Math.\&Stats. Inst., Federal Univ. of Goiás, Goiânia, Brazil. \texttt{ragarcia@ufg.br}}
\author{Boris Odehnal}
\thanks{B. Odehnal, Geom. Dept., Univ. Applied Arts, Vienna, Austria.
\texttt{boris.odehnal@uni-ak.ac.at}}
\author{Dan Reznik$^*$}
\thanks{D. Reznik$^*$, Data Science Consulting Ltd., Rio de Janeiro, Brazil. \texttt{dreznik@gmail.com}}

\begin{document}

\maketitle
\begin{abstract}
We analyze loci of triangle centers over variants of two-well known triangle porisms: the bicentric and confocal families. Specifically, we evoke the general version of Poncelet's closure theorem whereby individual sides can be made tangent to separate in-pencil caustics. We show that despite the more complicated dynamic geometry, the locus of certain triangle centers and associated points remain conics and/or circles. 
\end{abstract}


\section{Introduction}
\cref{fig:poncelet}(left) shows what is actually a special case of Poncelet's porism: if an $N$-gon can be found with all vertices on a first conic and all sides tangent to a second one (the ``caustic''), then a one-parameter family of said $N$-gons exists, e.g., with vertex on any point on the first conic \cite{dragovic11}. \cref{fig:poncelet}(right) illustrates the general case, called ``Poncelet's Closure Theorem'' (PCT), contemplating a porism with multiple caustics. It can be stated as follows \cite{centina15}: 

\begin{figure}
    \centering
    \includegraphics[trim=0 75 0 115,clip,width=\textwidth]{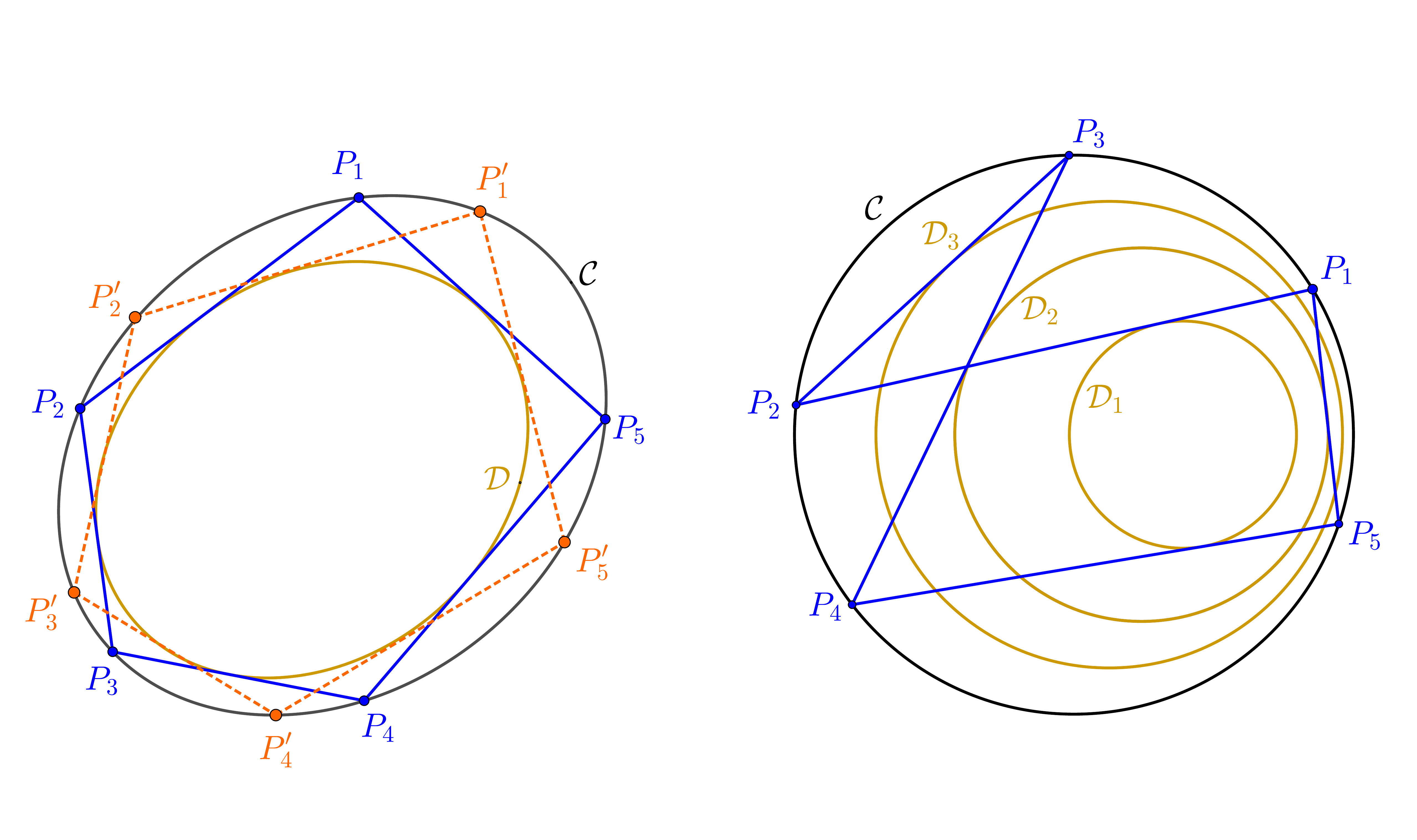}
    \caption{\textbf{Left:} A Poncelet porism of 5-gons interscribed between two ellipses $\C$ and $\D$. \href{https://youtu.be/kzxf7ZgJ5Hw}{Video}. \textbf{Right:} Poncelet's Closure Theorem (PCT) \cite{centina15} contemplates a porism of polygons inscribed in an outer conic $\C$ with sides tangent to one or more in-pencil conics $\D_i$. Shown here is a pencil of coaxial circles. \href{https://youtu.be/L5A_S4VQLiw}{Video}}
    \label{fig:poncelet}
\end{figure}

\begin{theorem}[PCT]
Let $\C$ and $\D_i$, $i=1,\ldots,M$ be $M+1$ distinct conics in the same linear pencil. If an $N$-gon can be constructed with all vertices on $\C$ such that each side is tangent to some $\D_i$, a porism exists of such $N$-gons.
\end{theorem}

Recall the pencil of two conics is a linear combination of their implicit equations \cite{bix}.

%\begin{figure}
%    \centering
%    \includegraphics[width=\textwidth]{pics/0001_del_centina.pdf}
%    \caption{This figure, which motivated this paper, is borrowed from \cite[Fig. 7]{centina15}. \textbf{Left:} the ``standard'' Poncelet's porism (left) of polygons inscribed in an outer conic with sides tangent to an inner one is actually a special case of a more general result, \href{https://youtu.be/kzxf7ZgJ5Hw}{Video}. \textbf{Right:} Poncelet's Closure Theorem (PCT) \cite{centina15} contemplates a porism of polygons inscribed in an outer conic $\C$ with sides tangent to one or more conics $\D_i$, all of which form a pencil with $\C$. The scene is shown as a pencil of circles, which (for most cases) is projectively equivalent to a pencil of conics.}
%\label{fig:del-centina}
%\end{figure}

\begin{figure}
    \centering
    \includegraphics[width=\textwidth]{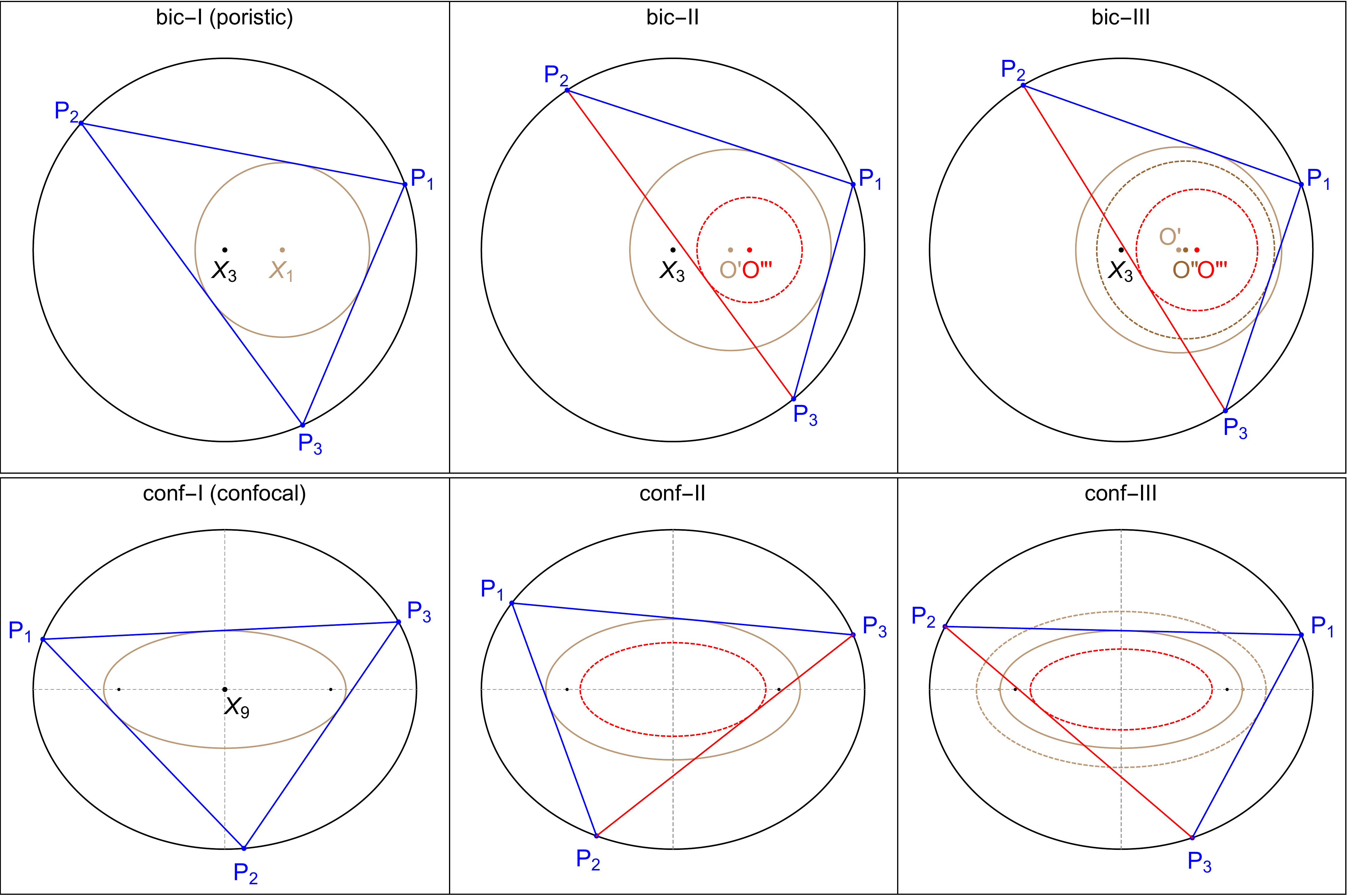}
    \caption{Triangle families considered herein, whose names appear at the top (bic-I, bic-II, etc.). The roman numeral after each name is the caustic count. \textbf{Top row}: the poristic pair (left) and its derivatives with 2, and 3 in-pencil caustics (middle and right).
    \textbf{Bottom row:} the confocal pair (left) and its derivatives with 2, and 3 in-pencil caustics (middle and right). \href{https://youtu.be/8HXgkuY-nFQ}{Video}}
    \label{fig:bics-confs}
\end{figure}

Referring to \cref{fig:bics-confs}, with PCT in mind, we analyze loci of triangle centers over variants of two well-known triangle porisms:

\begin{compactitem}
    \item the bicentric family (also known as poristic triangles or Chapple's porism): these have a fixed incircle and circumcircle; see \cref{fig:poristic}. 
    \item the confocal family (also known as the elliptic billiard): a constant-perimeter family of triangles interscribed between two confocal ellipses; see \cref{fig:inc-exc}.
\end{compactitem}

Specifically, we want to analyze loci of their triangle centers as we perturb their $N=3$ caustic and/or add additional caustics, in the spirit of PCT. The main surprise is that the locus of certain triangle centers and excenters ``survive'' (remain conics) despite PCT.

\subsection*{Summary of the results:}

\begin{compactitem}
\item For the 2-caustic poristic family; see \cref{fig:bics-confs} (center top):
\begin{compactitem}
    \item \cref{thm:bicII-x1}: the locus of the incenter is still a circle.
    \item \cref{prop:bicII-x2}: the locus of the barycenter is an algebraic curve of degree 6 (given explicitly in \cref{app:locus-x2}). Indeed, in the standard 2-circle poristic family, this locus is a simple circle \cite{odehnal2011-poristic}.
    \item The locus of one of the excenters is also a circle while that of the other two are distinct non-conics.
    \item \cref{conj:bicII-stationary}: experimental evidence suggests that a necessary (though not sufficient) condition for the locus of a given triangle center to be a conic is that its locus over the poristic family is a point.
\end{compactitem}
\item For the 2-caustic confocal family; see \cref{fig:bics-confs} (center bottom):
\begin{compactitem}
    \item \cref{prop:confII-x1}: the locus of the incenter is an ellipse only for the standard confocal pair.
    \item \cref{thm:confII-exc}: The loci of two of the excenters are the same ellipse while the third one sweeps a non-conic.
    \item \cref{cor:confII-n4,cor:confII-n6}: the elliptic locus of said excenters can assume, for certain configurations, special shapes: (i) its aspect ratio can be the reciprocal of that of the external ellipse, and (ii) it can be a circle.
\end{compactitem}
\end{compactitem}

In \cref{sec:three-caustics} we look at either a circle- or ellipse-inscribed family with three caustics (right column of \cref{fig:bics-confs}). We find that (i) the locus of a triangle center is never a conic, and that (ii)  the shape of many loci are more complicated (see \cref{fig:conf-III}).

Loci phenomena for all families and triangle centers considered are summarized side-by-side in \cref{tab:summary} in \cref{sec:videos}.

\subsection*{Related work}

Loci of triangle centers over poristic triangles have been studied in depth in \cite{garcia2020-similarity-I,odehnal2011-poristic}. As shown there, the loci of triangle centers can be points, segments, conics, and other shapes.

Regarding the confocal family,  loci of both incenter and  excenters\footnote{The intersections of external bisectors \cite[Excenter]{mw}.} are ellipses with reciprocal aspect ratios \cite{garcia2019-incenter,olga14}. Loci of other notable centers such as the circumcenter, orthocenter, etc., are also elliptic  \cite{corentin2021-circum,garcia2019-incenter,garcia2020-ellipses}. Certain loci are non-conic  (e.g., that of the symmedian point, Fermat point, etc.)   \cite{garcia2020-new-properties}. Remarkably, the locus of the mittenpunkt\footnote{This is where lines from the excenters through the midpoints of a triangle concur \cite[X(9)]{etc}.} is stationary at the common center \cite{reznik2020-intelligencer}.

Experiments suggests that only in the confocal pair can the locus of the incenter be a conic \cite{helman2021-power-loci}.

\subsection*{Article structure}

The basic geometry of the poristic (resp. confocal) family is reviewed in \cref{sec:poristic} (resp. \cref{sec:confocal}). In \cref{sec:bicII} (resp. \cref{sec:confII}) we analyze loci when an additional in-pencil caustic is added.

Links to simulation videos are included in the caption of most figures (a table with all videos mentioned appears in \cref{sec:videos}. As mentioned before, \cref{app:locus-x2} provides an explicit parametrization for the locus of the barycenter under the bic-II family. \cref{app:vtx-param} provides an explicit parametrization for bic-II and conf-II vertices. Finally, \cref{app:symbols} compiles all main symbols used herein in a table.

\section{Review: Poristic Family}
\label{sec:poristic}
Referring to \cref{fig:poristic}, any triangle is always ``poristic'' with respect to its circumcircle and incircle, in the sense that a one-parameter Poncelet family of triangles automatically exists interscribed between said circle pair.

\begin{figure}
    \centering
    \includegraphics[width=.9\textwidth]{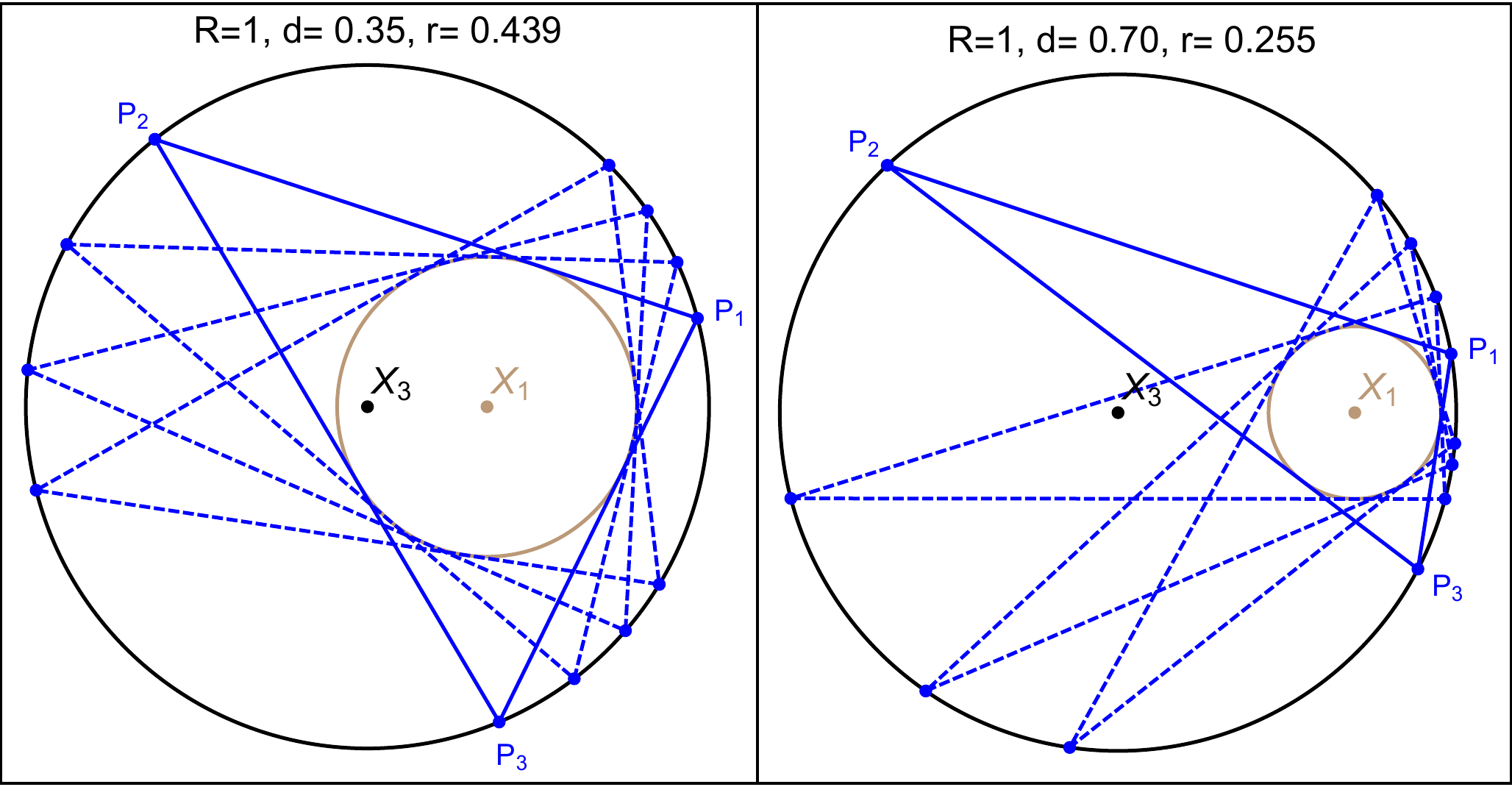}
    \caption{Two examples of the poristic family, i.e., a family of triangles (blue and dashed blue) inscribed in an outer circle (black) while circumscribed about an inner one (brown). Using Kimberling's notation \cite{etc}, $X_3$ (resp. $X_1$) denote the circumcenter (resp. incenter). \href{https://bit.ly/2V7lf9b}{Live}}
    \label{fig:poristic}
\end{figure}

This fact was first described by William Chapple in 1749 \cite{chapple1749}, almost 80 years prior to Poncelet generalizing it to any pair of conics in 1822 \cite{poncelet1822}. Let $d,R,r$ denote, respectively, the distance between circle centers, the circumradius and inradius. Chapple derived\footnote{Also known as ``Euler's Theorem'', though only published in 1765.}
\begin{equation}
    d^2 = R(R-2 r),\;\;\;R>2r.
    \label{eqn:chapple}
\end{equation}
Equivalently, $1/(R-d)+1/(R+d)=1/r$. The abovementioned porism is sometimes called the $N=3$ {\em bicentric family}, i.e., a Poncelet porism between two non-concentric circles. We will refer to it as ``bic-I'' for reasons that will become clear.

Adopting Kimberling's $X_k$ notation \cite{etc} for triangle centers, the incenter $X_1$ and circumcenter $X_3$ of bic-I triangles are fixed by definition. Indeed, it has been shown that several triangle centers remain stationary over this family, while others sweep circles  \cite{odehnal2011-poristic}. Another interesting fact is that this family conserves the sum of its interior angle cosines\footnote{This conservation is also manifested by the {\em confocal} pair, seen later in the article.}, since \cite[Inradius, Eqn. 9]{mw}:
\[ \sum_{i=1}^{3}\cos\theta_i = 1+\frac{r}{R}.\]

\section{Bicentrics with Two Caustics}
\label{sec:bicII}
We now consider a slight variant of the poristic family, namely triangles $\T=P_1 P_2 P_3$, inscribed in an outer circle $\C=[(0,0),R]$, with two sides $P_1 P_2$ and $P_1 P_3$ tangent to an inner circle $\C'=[(d,0),r]$, where \cref{eqn:chapple} does not hold.

\begin{definition}
The pencil of two conics $\E$ and $\E'$ is a one-parameter family of the linear combinations of their implicit equations. 
\end{definition}

We review a result, called Poncelet's Main Lemma (PML), which will simplify our constructions. This was proved by Poncelet so as to support the proof of PCT \cite{centina15,poncelet1822}:

\begin{lemma}[PML]
Let $\C$ and $\D_i$, $i=1,\ldots,N-1$ be $N$ distinct conics in the same pencil. Let $P_i$, $i=1,\ldots,N$ denote the vertices of a polygonal chain inscribed in $\C$ whose $N-1$ sides are each tangent to a $\D_i$. So $P_1$ (resp. $P_N$) is the first (resp. last) vertex in the chain. The open side $P_N P_1$ will envelop a conic $\D_N$, contained in the original pencil.
\label{lem:pml}
\end{lemma}

Referring to \cref{fig:bic-II}, \cref{lem:pml} implies that the envelope of $P_2 P_3$ will be a distinct circle $\C'''$ in the pencil of $\C,\C'$, which can be regarded as a second caustic. We therefore call this family ``bic-II'' (bicentric family with two caustics).

\begin{figure}
    \centering
    \includegraphics[width=\textwidth]{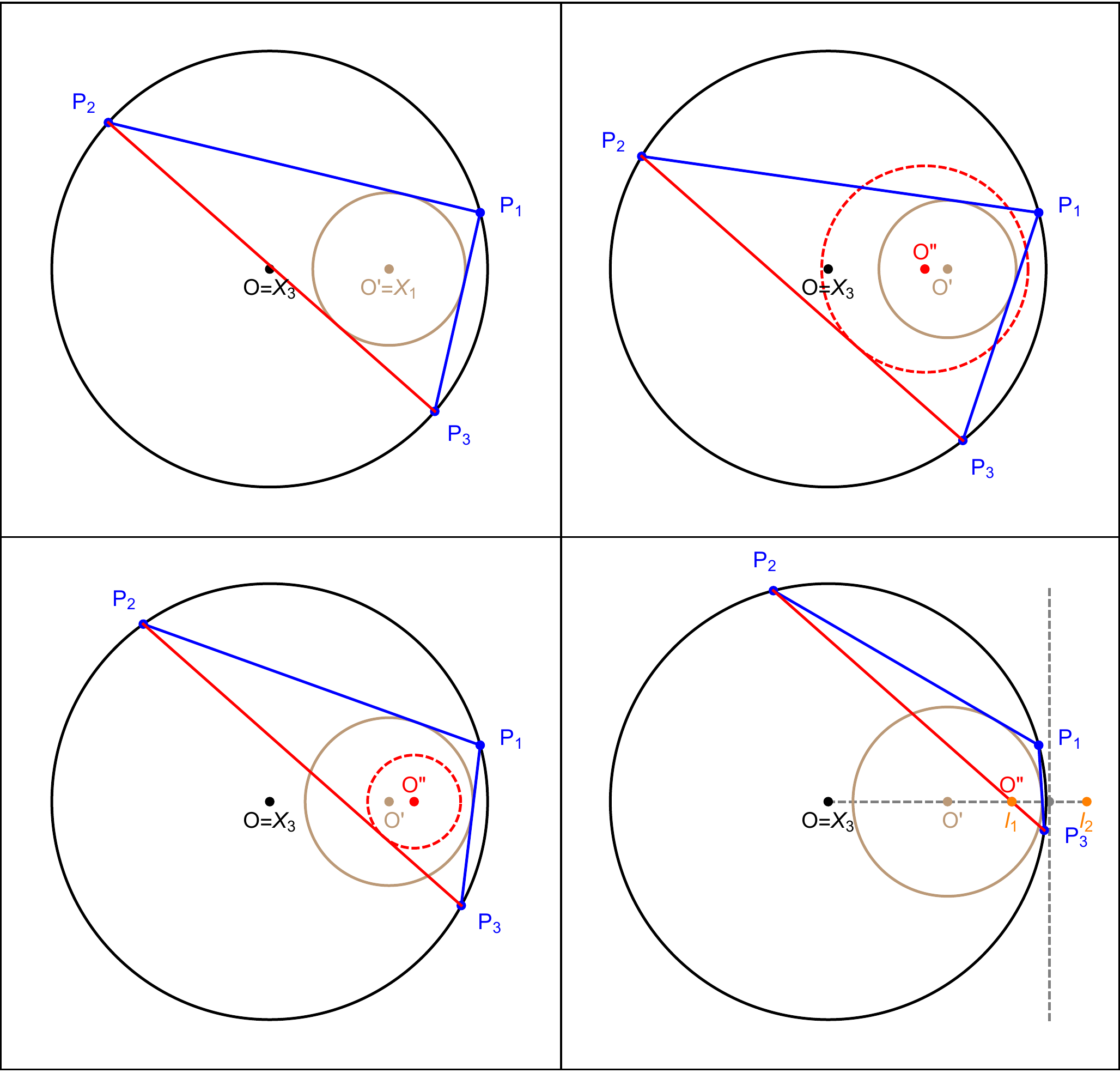}
    \caption{\textbf{Top left:} Chapple's porism or ``bic-I'' (poristic) family: triangles $\T=P_1 P_2 P_3$ with a fixed circumcircle $\C$ (black) and incircle $\C'$ (brown), centered at $O=X_3$ and $O'=X_1$ and with radii $R$ and $r$, respectively. \textbf{Top right:} the bic-II family is obtained by decreasing $r$  (bic-II family) while maintaining $P_1 P_3$ and $P_1 P_2$ tangent to $\C'$. Poncelet's general closure theorem predicts that the envelope of the third side (red) is a  third circle $\C''$ (dashed red) in the pencil of $\C$ and $\C'$ centered at $O''$. Note $X_1$ (not shown) will no longer coincide with $O'$. \textbf{Bottom left:} a bic-II triangle where $\C'$ is exterior to the envelope of the third side (dashed red). \textbf{Bottom right:} a bic-II triangle such that the envelope of the third side collapses to a limiting point $\ell_1$ of the $\C,\C'$ pair. Also shown are (i) the radical axis (vertical dashed gray) of $\C,\C',\C''$, and (ii) the second limiting point $\ell_2$ of the pencil. \href{https://youtu.be/OM7uilfdGgk}{Video}}
    \label{fig:bic-II}
\end{figure}

Manipulation with CAS yields:

\begin{proposition}
Over the bic-II family, the envelope $\C''=[O'',r'']$ of $P_2 P_3$ is a circle in the pencil of $\C,\C'$ given by
\begin{align*}
O''=  &\left[\frac{4 d R^2 r^2}{(R^2 - d^2)^2}, 0\right], \\
r'' = &\frac{R(R^4 - 2 R^2 d^2 - 2 R^2 r^2 + d^4 - 2 d^2 r^2)}{(R^2-d^2)^2} = \frac{(p^2q^2 - p^2 - q^2)(p + q)d}{p^2q^2(p - q)},
\end{align*}
where $p=(R+d)/r$ and $q=(R-d)/r$.
\end{proposition}
\noindent Note that $r'' = 0$ is achieved if $(R^2-d^2)^2- 2r^2(R^2+d^2) = 0$.
% \[ R^4 - 2 R^2 d^2 - 2 R^2 r^2 + d^4 - 2 d^2 r^2 = 0 \]

Recall that for even N, the diagonals of Poncelet N-periodics in the bicentric pair (with incircle and circumcircle) meet at one of the limiting points of the pair, see \cref{fig:bic-II} (bottom right). The above relation is equivalent to the well-known condition for a pair of circles to admit a Poncelet family of quadrilaterals. One formulation, due to Kerawala \cite[Eq. 39]{mw} is
\[ \frac{1}{(R-d)^2} + \frac{1}{(R+d)^2} = \frac{1}{r^2} \cdot \]

\begin{theorem}
Over the bic-II family, the locus of the incenter $X_1$ is a circle $\C_1=[O_1,r_1]$ not in the pencil of $\C,\C'$, given by
\[ O_1= \left[\frac{2 d R r}{R^2 - d^2},0\right],\;\;\;r_1=\frac{R (R^2 - 2 R r - d^2)}{R^2 - d^2} \cdot
\]
\label{thm:bicII-x1}
\end{theorem}
Referring to \cref{fig:bicII-x1-proof}, the following proof was kindly contributed by Alexey Zaslavsky and Arseniy Akopyan, and later adapted by Mark Helman \cite{helman2021-private,zaslavsky2021-private}. 

\begin{figure}
    \centering
    \includegraphics[width=.6\textwidth,frame]{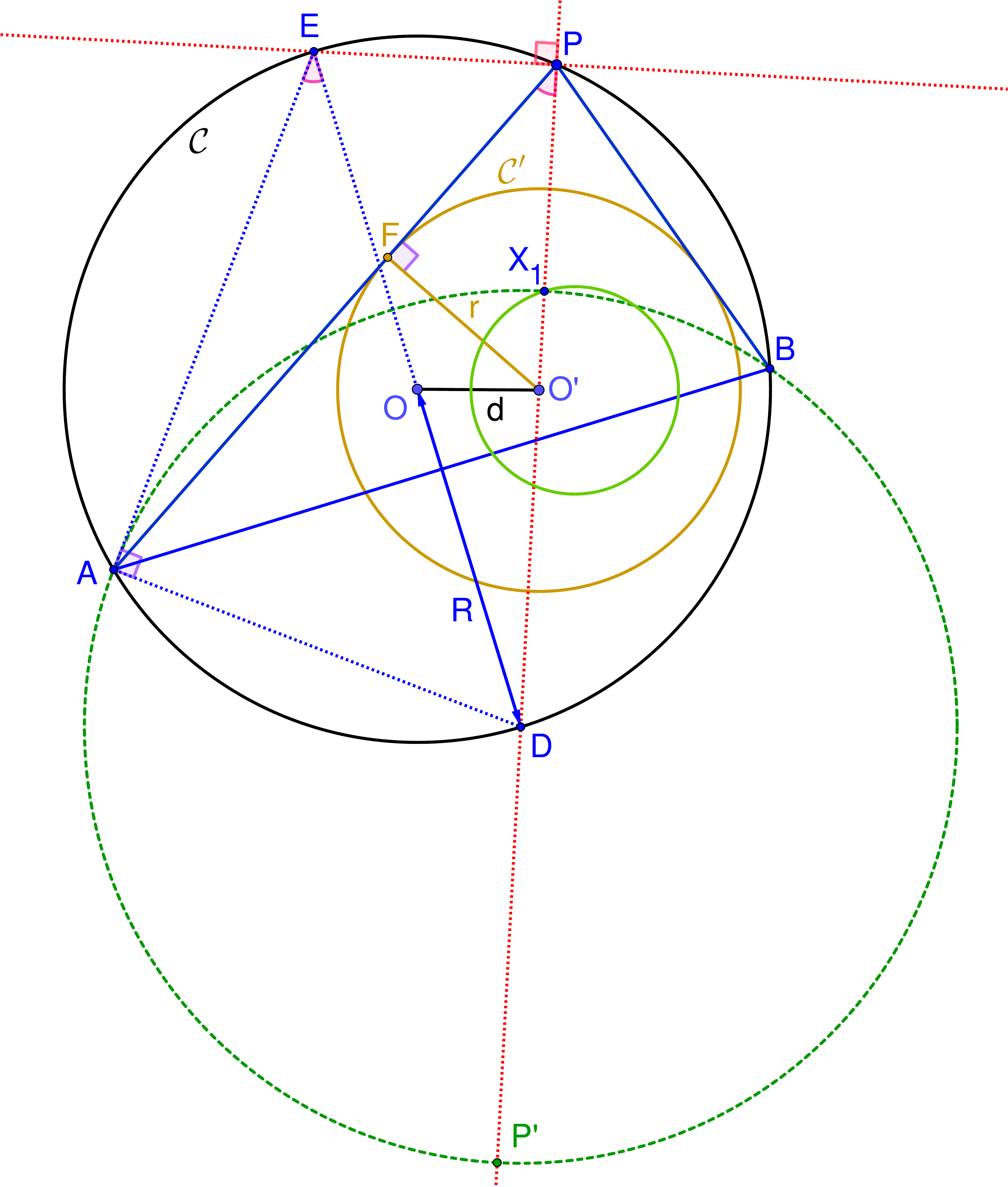}
    \caption{Objects and labels used in the proof of \cref{thm:bicII-x1}. A bic-II triangle $PAB$ is shown, inscribed in an outer circle $\C=(O,R)$ (black) and with two sides tangent to an internal circle $\C'=(O',r)$ (brown). The distance between their centers is $d$. The ``incenter-excenter'' circle (dashed green) is centered at the intersection of the chord $PO'$ with $\C$ and contains $X_1,P'$ as antipodes as well as tangent chord endpoints $A,B$. Triangles $E A D$ and $P F O'$ are similar. Also shown (green) is the circular locus of $X_1$ over the bic-II family. Original proof by Alexey Zaslavsky and Arseniy Akopyan, adapted by Mark Helman: \href{https://www.geogebra.org/classic/zm7wgnrz}{GeoGebra}}
    \label{fig:bicII-x1-proof}
\end{figure}

 \begin{proof}
Let $\T=P A B$ be a triangle inscribed to an outer circle $\C=(O,R)$ with sides $P A$ and $P B$ tangent to an inner circle $\C'=(O',r)$. Let $d=|O-O'|$.

Let $X_1$ denote the incenter of $\T$. Let the interior and exterior angle bisectors at $P$ intersect $\C$ at $D$ and $E$, respectively. Let $F$ be the point of contact of $P A$ with $\C'$.

Since the quadrilateral $AEPD$ is cyclic, $\angle A E D=\angle A P D$. Also, $\angle E A D=\pi-\angle E P D=\pi-\pi/2=\pi/2=\angle P F O'$. Thus, triangles $E A D$ and $P F O'$ are similar, therefore $A D/D E= F O'/O'P$, so $A D/(2R)=r/O'P$. By the Trillium Theorem \cite[Incenter Excenter Circle]{mw}, $AD=D X_1$, so $D X_1=2Rr/O'P$.

Using the definition of the power of a point with respect to circle $\C=EPBDA$, $(O'P)(O'D)=R^2-d^2$, so $O'P=(R^2-d^2)/O'D$. Substituting, we have $D X_1=2Rr(O'D)/(R^2-d^2)$, so $DX_1/O'D=2Rr/(R^2-d^2)$. 
Finally,
\[ \frac{O'X_1}{O'D}=\frac{DX_1-O'D}{O'D}=\frac{DX_1}{O'D}-1=\frac{2Rr}{R^2-d^2}-1.\]

The above implies $O'X_1 / O'D$ is constant and independent of $\T$. This shows that there is a homothety with center $O'$ that sends $\C$ to the locus of $X_1$, so the claim follows. The actual values of $O_1$ and $r_1$ are obtained by applying said homothety to $O$ and $R$ of $\C$.
\end{proof}

\noindent Note that if \cref{eqn:chapple} holds, $r_1=0$, i.e., the family is the standard poristic one for which the incenter locus is a point.

Let $\T'$ denote the {\em excentral triangle} of $\T$ with sides through the vertices of $\T$, perpendicular to the angle bisectors. The vertices of $\T'$ are known as the {\em excenters} $P_i'$, $i=1,2,3$.

In \cite{etc}, the circumcenter (resp. barycenter) of $\T'$ is named the Bevan point $X_{40}$ (resp. $X_{165}$). It is known that $X_{40}$ (resp. $X_{165}$) is a reflection of $X_1$ about $X_3$ with a scale of 1 (resp. 1/3). Referring to \cref{fig:x1-x40-x165}:
\begin{corollary}
Over the bic-II family, the locus of $X_{40}$ (resp. $X_{165}$) is a circle $\C_{40}=[O_{40},r_{40}]=[-O_1,r_1]$ (resp. $\C_{165}=[-O_1/3,r_1/3]$).
\end{corollary}

\begin{corollary}
Over the bic-II family, the locus of any triangle center which is at a fixed proportion from $X_1$ and $X_3$ will be a circle.
\end{corollary}

As an example, consider the circumcircle-inverse of $X_1$ (resp. $X_{40}$), called $X_{36}$ (resp. $X_{2077}$) on \cite{etc}. Since the circumcircle is fixed, both their loci are automatically circles.

\begin{figure}
    \centering
    \includegraphics[width=\textwidth]{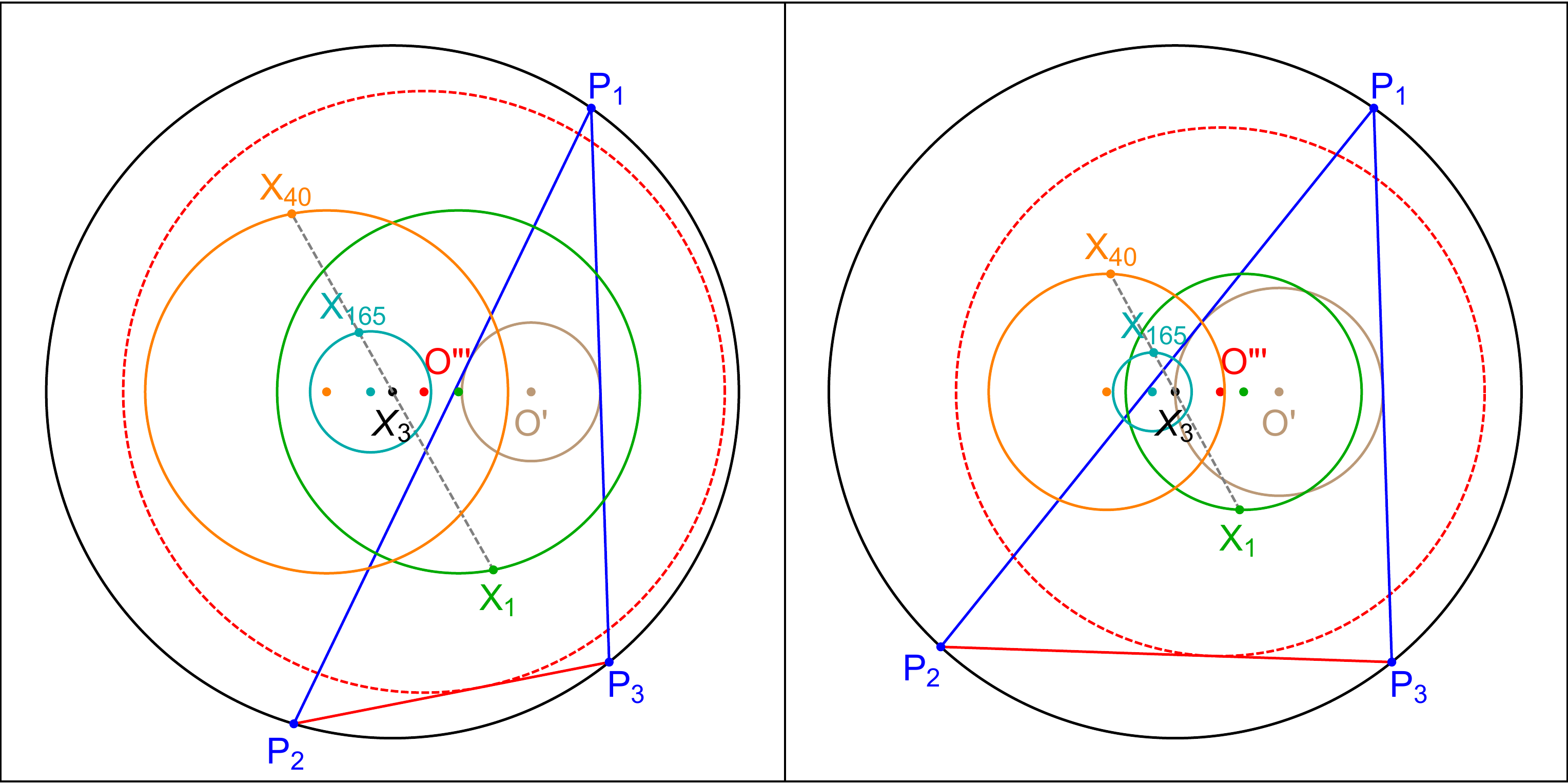}
    \caption{\textbf{Left:} A bic-II triangle $P_1 P_2 P_3$ inscribed in an outer (black) circle centered on $X_3$. Sides $P_1 P_2$ and $P_1 P_3$ (blue) are tangent to a first (brown) circle, centered on $O'$. Over the family, the third side (red) envelops a second, in-pencil circle (dashed red) centered on $O_\text{env}$. The loci of $X_1$ (green) and $X_{40}$ (orange) are copies of each other reflected about $O$. The locus of $X_{165}$ (light blue) is a copy of that of $X_1$, $1/3$-scaled about $O$. Notice $X_k,k=$1, 3, 40, 165 are collinear. \textbf{Right:} a similar setup but with an increase in the first caustic radius and a reduction in the distance between the centers, showing all loci remain circular. \href{https://youtu.be/qJGhf798E-s}{Video}}
    \label{fig:x1-x40-x165}
\end{figure}

%LINE OI Line: X(1)X(3) contains X(40) and X(165): L(650) Trilinear Polar of X(651)     OI Line

%\textcolor{red}{ron: X(5537) = circumcircle-inverse of X(165)}

Referring to \cref{fig:bicII-topo-x2}, CAS manipulation yields:

\begin{proposition}
Over the bic-II family, the locus of the barycenter $X_2$ is an algebraic curve of degree 6.
\label{prop:bicII-x2}
\end{proposition}
\cref{app:locus-x2} gives an explicit expression for this locus.

\begin{figure}
    \centering
    \includegraphics[width=\textwidth]{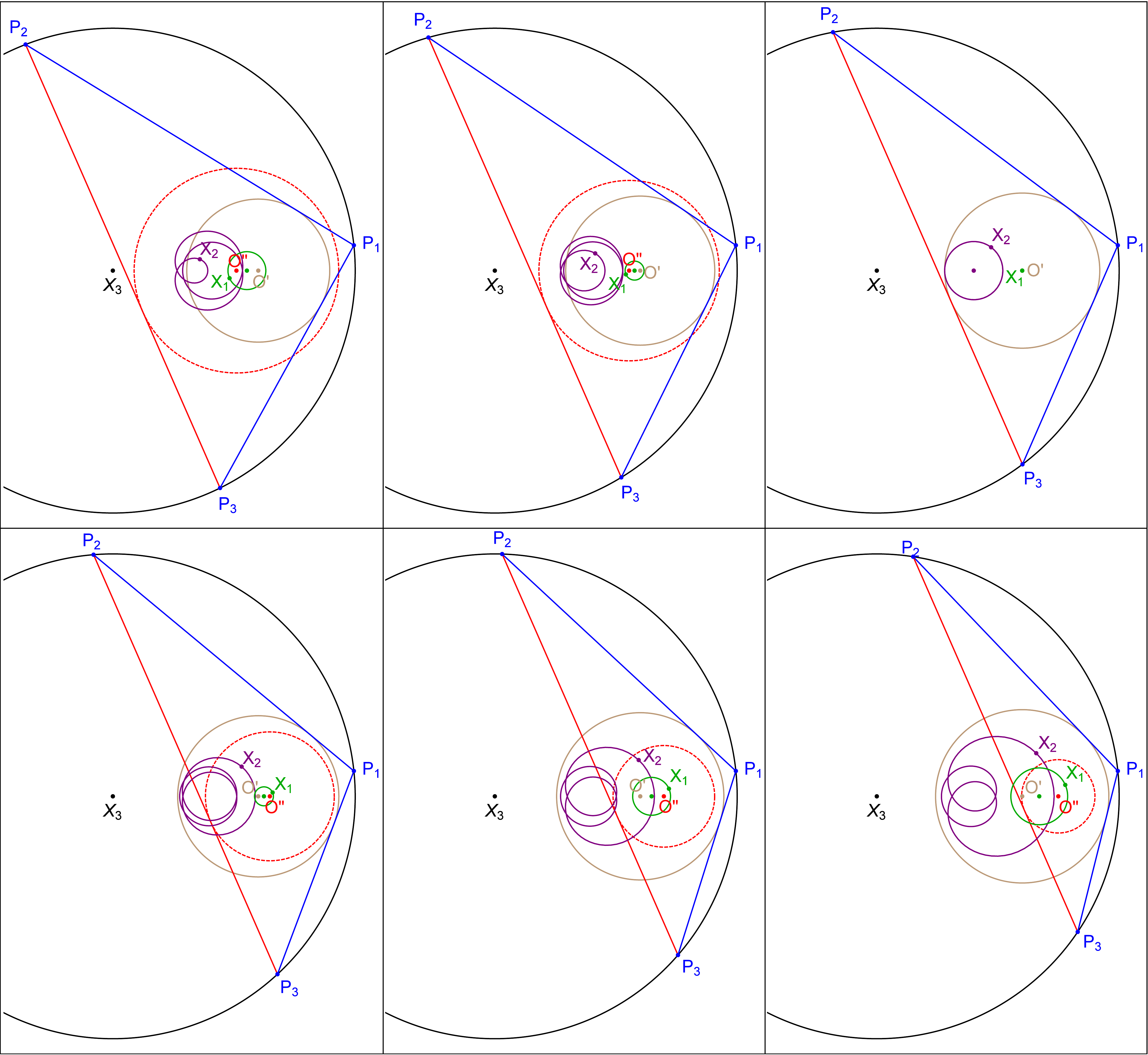}
    \caption{Sextic loci of $X_2$ under fixed $R,d$ and with increasing $r$. Notice that at the poristic configuration (top right), the locus of $X_2$ collapses to a circle (swept three times). \href{https://youtu.be/3dnsWPlAmxE}{Video}}
    \label{fig:bicII-topo-x2}
\end{figure}

The non-conic loci of $X_k,k=$2,4,5,6 are shown in \cref{fig:bicII-x2x4x5}.

\begin{figure}
    \centering
    \includegraphics[trim=0 50 0 0,clip,width=\textwidth]{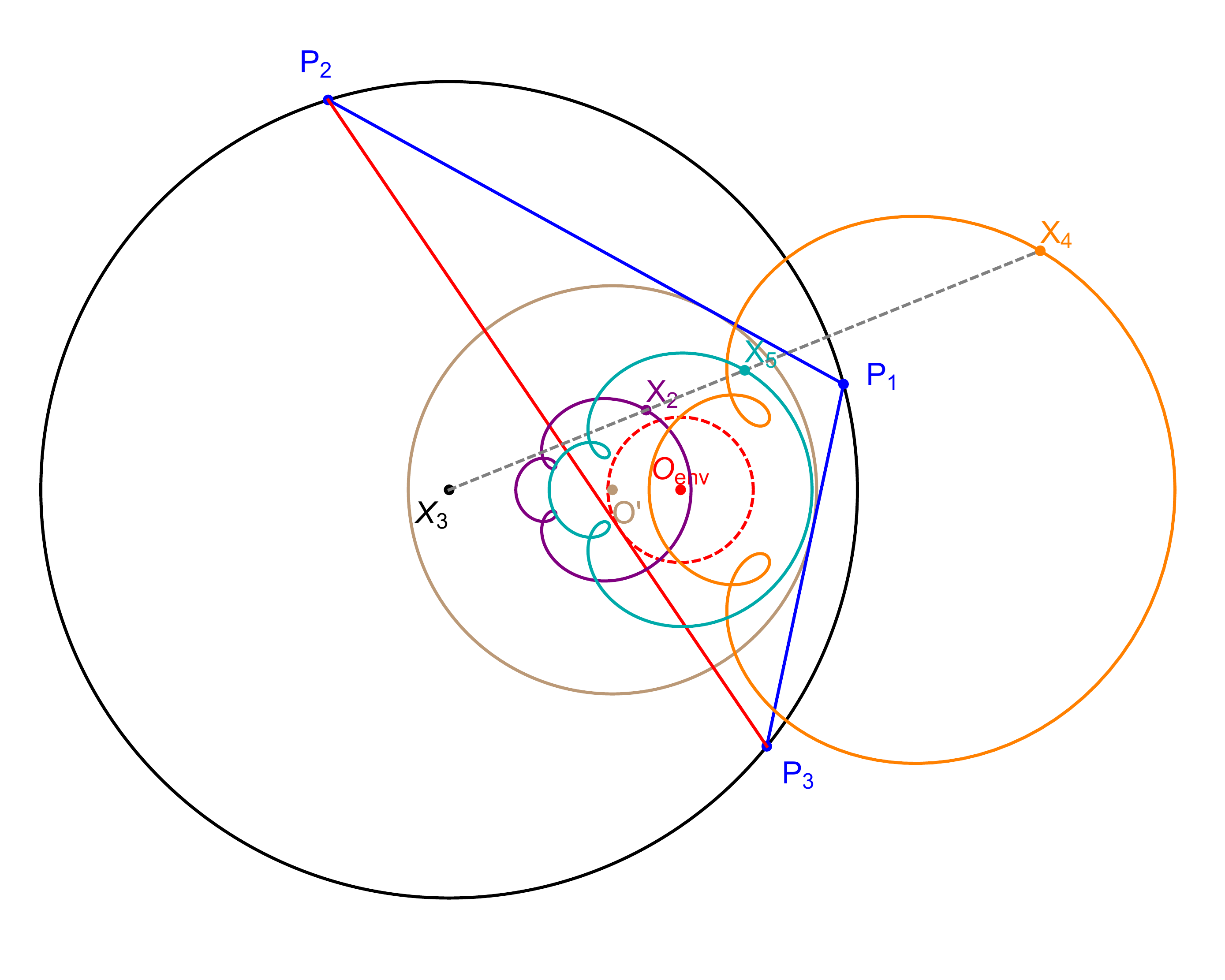}
    \caption{A bic-II triangle and the sextic loci of $X_k,k=$2, 4, 5. Along with $X_3$, these are collinear on the Euler line (dashed black) and at fixed proportions to $X_3 X_2$. Therefore, the three loopy loci are homothetic. \href{https://youtu.be/6Fqp6Z1Q-0A}{Video}}
    \label{fig:bicII-x2x4x5}
\end{figure}

Recall that some triangle centers with the property to always lie on a triangle's circumcircle, e.g., $X_k,k$=74, 98, 99, 100, 101, etc., see \cite[Circumcircle]{mw} for a larger list. Let $X$ be a triangle center not on this list. Recall also that over the poristic family, several triangle centers were identified as stationary \cite{odehnal2011-poristic}; see the first row on \cref{tab:poristic-conics}. Experiments suggest that:

\begin{conjecture}
Over the bic-II family, a necessary (though not sufficient) condition for the locus of $X$ to be a conic is that its locus over the bic-I (poristic) family is a stationary point.
\label{conj:bicII-stationary}
\end{conjecture}

Referring to \cref{tab:poristic-conics}, examples of triangle centers stationary over the bic-I family which do not yield conics over the bic-II family include $X_k,k=$46, 56, 57, 65.

% on OI line  (necessary?) and stationary over bic-I
% 1, 3, 35, 36, 40, 46, 55, 56, 57, 65, 165, 354, 484, 942,
% on OI but not stationary over bic-I
% 171, 241, 260, 517, 559, 940,  980, 982, 986, 988, 999

\begin{table}
\begingroup
\setlength{\tabcolsep}{2pt} %
\begin{tabular}{|l||
c|c|c|c|c|c|c|c|c|c|c|c|c|c|}
\hline
\makecell[lc]{bic-I\\point}& $X_1$ & $X_3$ & $X_{35}$ & $X_{36}$ & $X_{40}$ & $X_{46}$ & $X_{55}$ & $X_{56}$ & $X_{57}$ & $X_{65}$ & $X_{165}$ & $X_{354}$ & $X_{484}$ & $X_{942}$ \\
\hline
bic-II & C & P & E & E & C & X & E & X & X & X & C & E & E & E \\
bic-III & X & P & X & X & X & X & X & X & X & X & X & X & X & X\\
\hline
\end{tabular}
\endgroup
\caption{The first row lists triangle centers known to be stationary over bic-I (poristic family) amongst the first 1000 centers on \cite{etc}. The second (resp. third) row lists their locus type in the bic-II (resp. III) family. P,C,E,X stand for point, circle, ellipse, and non-conic, respectively.}
\label{tab:poristic-conics}
\end{table}

\subsection*{Studying the excenters}

Referring to \cref{fig:bicII-x1-proof}, notice the excenter $P'$ opposite to $C$ is antipodal to $X_1$ on the ``incenter-excenter'' circle centered on $D$. Therefore, a consequence of \cref{thm:bicII-x1} is that:

\begin{corollary}
Over the bic-II family, the locus of $P_1'$ is a circle $\C_1'=[-O_1,r_1']$ not in the pencil of $\C,\C'$, where
\[ r_1'=\frac{R(R^2 + 2 R r - d^2)}{R^2 - d^2} \cdot \]
\label{cor:bicII-exc}
\end{corollary}
It can be shown via CAS that the two other excenters sweep a sextic. This is illustrated in \cref{fig:bicII-excs}: let $P_i'$ denote the excenter opposite to $P_i$. Though the loci of $P_2'$ and $P_3'$ are distinct and non-conic, as shown above, that of $P_1'$ is a circle. Note that as derived in \cite{odehnal2011-poristic}, when \cref{eqn:chapple} holds (poristic family), $r_1'=2R$.

\begin{figure}
    \centering
    \includegraphics[width=\textwidth]{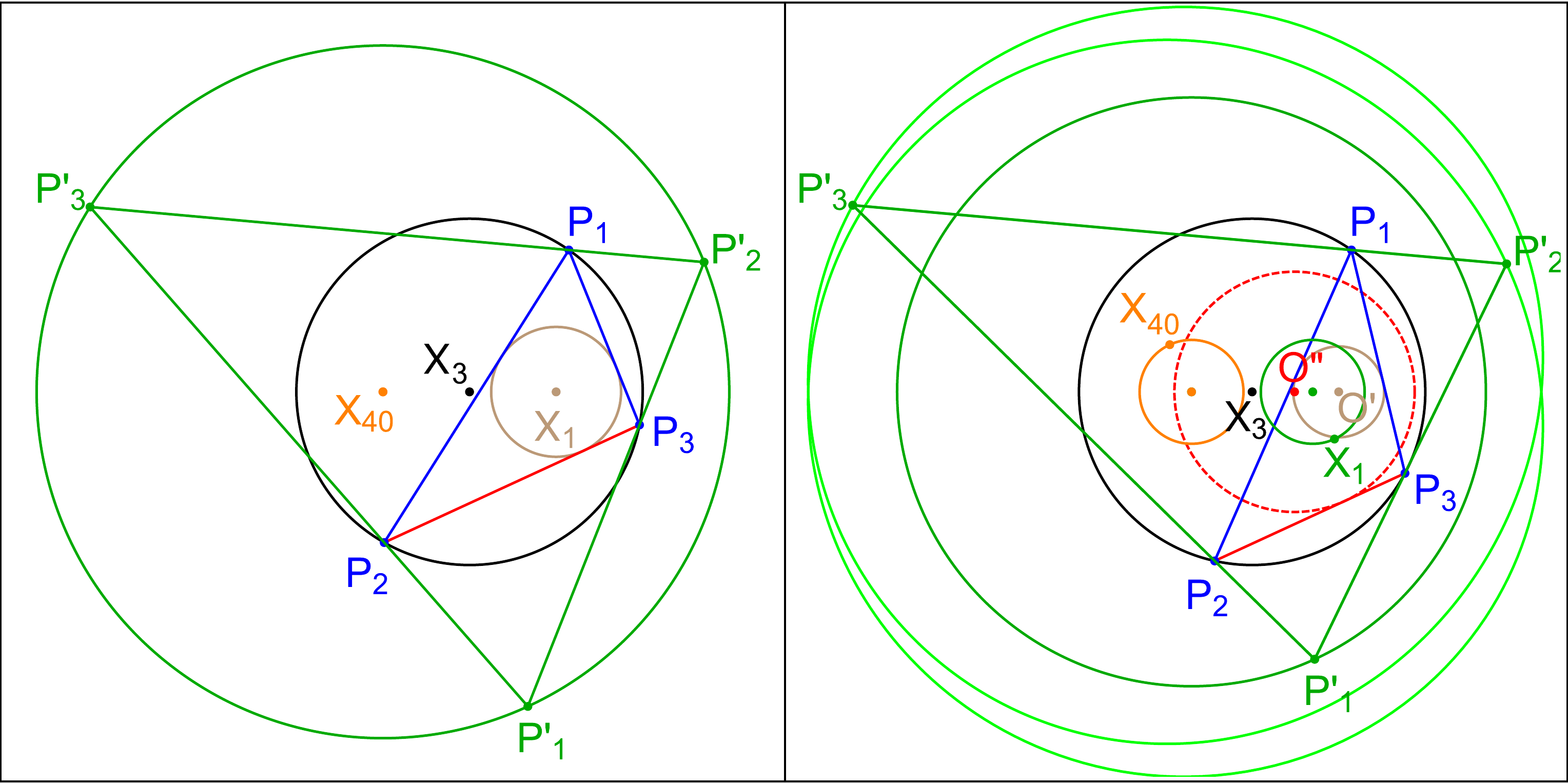}
    \caption{\textbf{Left:} a bic-I (poristic triangle $P_1 P_2 P_3$ and its excentral triangle $P_1' P_2' P_3'$ (dark green). The locus of its vertices is a circle centered on $X_{40}$ (stationary) and with radius twice that of the circumcircle of the original family. \textbf{Right:} a bic-II triangle and its excentral triangle (dark green). The locus of $X_{1}$ (green) and $X_{40}$ (orange) are identical circles (reflected about $X_3$). The locus of $P_1'$ (opposite to $P_1$) is a circle (dark green) concentric with the locus of $X_{40}$. Excenters $P_2'$ and $P_3'$ sweep non-conics (light green) which intersect on the $X_1 X_3$ line. \href{https://youtu.be/qdqIuT-Qk6k}{Video}}
    \label{fig:bicII-excs}
\end{figure}

\section{Review: Confocal Triangle Family}
\label{sec:confocal}
Referring to \cref{fig:inc-exc}, recall a special family of triangles interscribed between a pair of {\em confocal} ellipses $\E,\E'$. Let $a,b$ and $a',b'$ denote their semi-axes. The Cayley condition for any concentric, axis-parallel pair of ellipses to admit a 3-periodic family is $a'/a+b'/b=1$ \cite{cayley1857}. Let $\delta=\sqrt{a^4-a^2b^2+b^4}$ and $c^2=a^2-b^2$. As derived in \cite{garcia2019-incenter}, imposing confocality yields
\begin{equation}
a'=\frac{a\left(\delta-{b}^{2}\right)}{c^2},\;\;\;
b'=\frac{b\left({a}^{2}-\delta\right)}{c^2} \cdot
\label{eqn:confocal}
\end{equation}

\begin{figure}
    \centering
    \includegraphics[width=\textwidth]{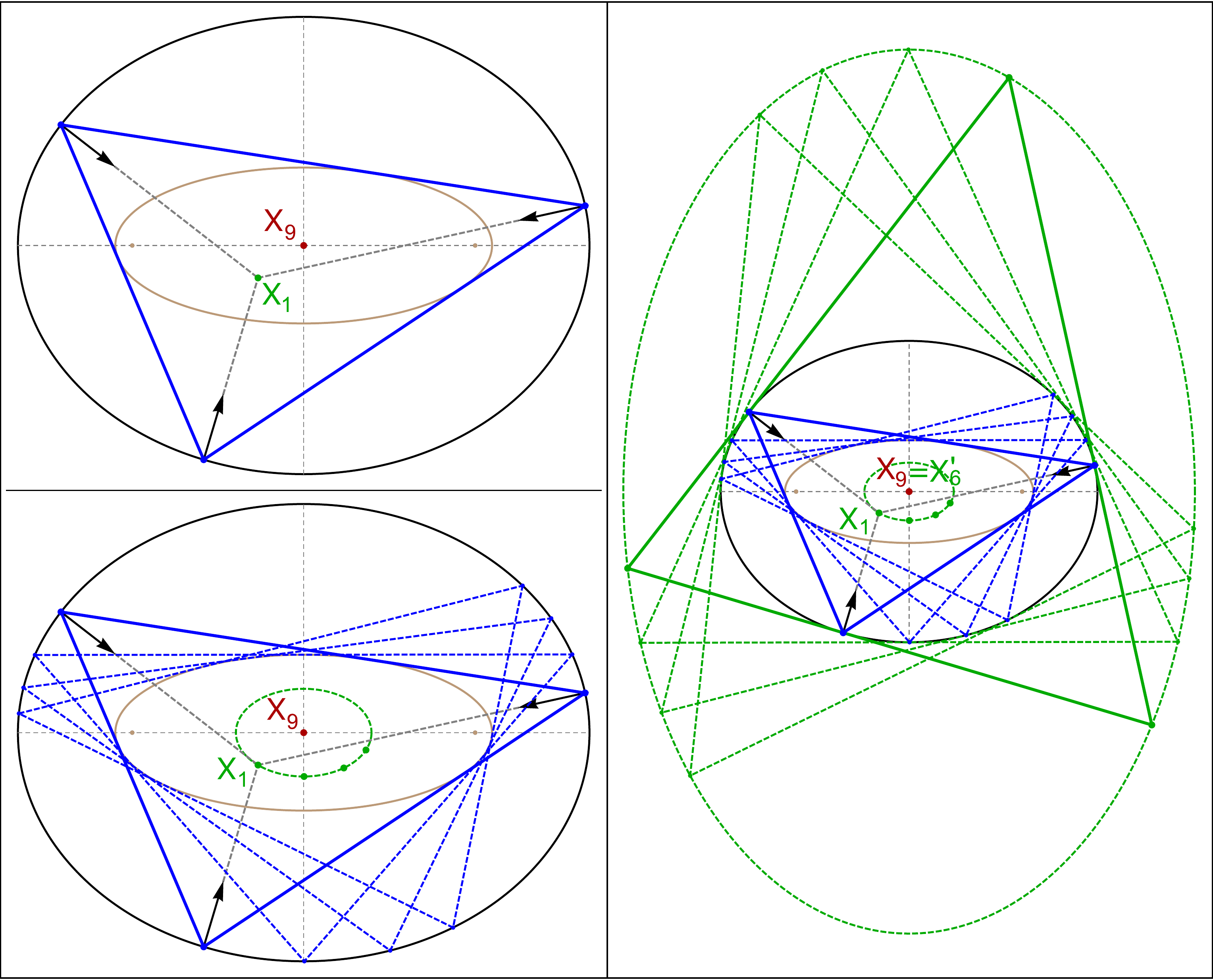}
    \caption{\textbf{Top Left:} A 3-periodic (solid blue) in the conf-I family. Owing to a special property of the confocal pair, internal angles are bisected by ellipse normals (black arrows); these concur at the incenter $X_1$. \textbf{Bottom Left:} superposition of 3-periodics (note: all have the same periemeter). Over the family, (i) the mittenpunkt $X_9$ is stationary and (ii) the incenter $X_1$ sweeps an ellipse (dashed green). \textbf{Right:} the family of excentral triangles (solid green) obtained from conf-I 3-periodics is itself a Ponceletian family. Its vertices (the excenters) sweep a concentric, axis-aligned ellipse (dashed green), whose aspect ratio is the reciprocal of the original incenter locus \cite{garcia2019-incenter}.  \href{https://bit.ly/368BuoM}{Live}}
    \label{fig:inc-exc}
\end{figure}

Referring to \cref{fig:inc-exc} (top left), this family has the special property that normals to $\E$ at the vertices are the bisectors, so triangles can be thought of as billiard (3-periodic) trajectories. Indeed, this system is also known as the $N=3$ {\em elliptic billiard} (see \cite{sergei91} for an introduction of its properties). For consistency with the previous section, we shall call this family ``conf-I'' (confocal family with a single caustic).

\cref{fig:inc-exc} depicts several interesting phenomena concerning conf-I triangles: (i) the mittenpunk $X_9$ is stationary \cite{reznik2020-intelligencer}; (ii) the ratio $r/R$ is conserved as is the sum of cosines \cite{garcia2020-new-properties}, (iii) the locus of the incenter $X_1$ is an ellipse concentric and aligned with $\E,\E'$ \cite{olga14,garcia2019-incenter}; (iv) the common locus of the excenters (vertices of the excentral triangle) is also a concentric, axis-aligned ellipse, whose aspect ratio is the reciprocal of that of (iii) \cite{garcia2019-incenter}. Let $a_1,b_1$ (resp. $a_e,b_e$) denote the semi-axes of (iii) and (iv), respectively. These are given by \cite{garcia2019-incenter}
\[ a_1=\frac{\delta-{b}^{2}}{a},\;\;\;
 b_1=\frac{{a}^{2}-\delta}{b},\;\;\;
 a_e=\frac{{b}^{2}+\delta}{a},\;\;\;
 b_e=\frac{{a}^{2}+\delta}{b} \cdot\]
\noindent Note that these two loci are ellipses with reciprocal aspect ratios, i.e., $a_1/b_1=b_e/a_e$ \cite{garcia2019-incenter}.

\section{Confocals with Two Caustics}
\label{sec:confII}
Referring to \cref{fig:confII}, we consider a variant of the conf-I family, namely triangles $\T=P_1 P_2 P_3$ inscribed in an outer ellipse $\E$ with semi-axes $a,b$. Sides $P_1 P_2$ and $P_1 P_3$ are tangent to an inner, confocal ellipse $\E'$ with semi-axes $a',b'$. In this case, \cref{eqn:confocal} in general does not hold. \cref{lem:pml} ensures that the envelope of $P_2 P_3$ will automatically be a conic in the pencil of $\E,\E'$, i.e., a second elliptic caustic. We therefore call such a system ``conf-II''. 

\begin{figure}
    \centering
    \includegraphics[width=\textwidth]{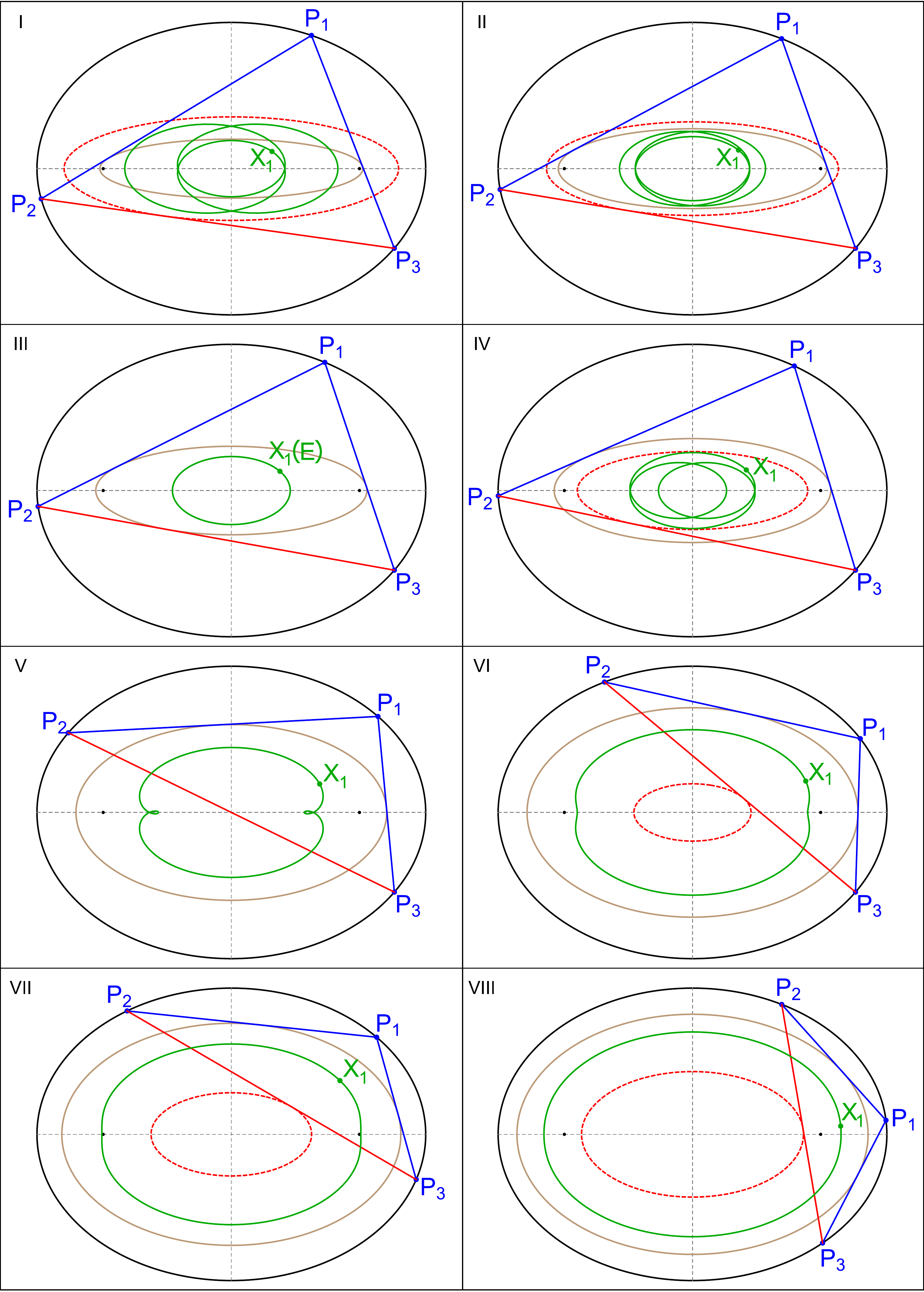}
    \caption{Left-to-right, top-to-bottom: the locus of $X_1$ (green) over conf-II triangles as the confocal caustic $\E'$ (brown) grows toward $\E$ (black). Also shown is the in-pencil envelope $\E''$ of $P_2 P_3$ (dashed red). At inset III, $\E''=\E$ and the locus of $X_1$ an ellipse. At inset V, $\E''$ is the center point. At inset VII, the locus becomes convex (see \cref{prop:confII-x1-convex}). \href{https://youtu.be/C14TL430UBc}{Video 1}, \href{https://youtu.be/kCY6KHFDV2M}{Video 2}}
    \label{fig:confII}
\end{figure}

\begin{remark}
It can be shown that the only non-degenerate conics in the pencil of two confocal ones $\E,\E'$ which are confocal with the pair are $\E$ and $\E'$ themselves. 
\end{remark}

\begin{remark}
It can be shown that over the conf-I family, the elliptic locus of $X_1$ and of the excenters does not belong to the pencil of $\E,\E'$.
\end{remark}

Poncelet's general theorem \cite{centina15} predicts that over the conf-II family, the envelope of $P_2 P_3$ is an ellipse $\E''$ which is concentric, axis-parallel, and in the pencil of $\E,\E'$. Via CAS, obtain:
\begin{proposition}
The semi-axes $a'',b''$ of $\E''$ are given by
\[a''= \frac{|a \zeta| }{ a^2b^2-c^2\lambda },\;\;\; b''=\frac{|b \zeta| }{ a^2b^2+c^2\lambda}, \]
where $\zeta= a^2b^2-(a^2 + b^2)\lambda$ and $\lambda=a^2-a'^2=b^2-b'^2$. \end{proposition}

\begin{remark}
In the conf-I family, the outer (resp. inner) ellipse is obtained when $\lambda=0$ (resp. $\lambda=a^2b^2( 2\delta-a^2-b^2)/c^4 $). Note that in this case, the envelope of $P_2 P_3$ is confocal (it is the inner ellipse itself).
\end{remark}

\begin{proposition}
Over the conf-II family, the locus of $X_1$ is only an ellipse in the conf-I configuration, i.e., the confocal pair $\E,\E'$ admits a 3-periodic family.
\label{prop:confII-x1}
\end{proposition}

\begin{proof} 
Consider the triangle $P_1P_2P_3$ parametrized in \cref{subsec:Conf-II}.
The locus $X_1(\lambda)$ is given by
\begin{gather*}
X_1(\lambda)=[x,y]= \frac{s_1P_1+s_2P_2+s_3P_3}{s_1+s_2+s_3},\\
s_1=|P_2-P_3|,\;\; s_2=|P_2-P_3|,\;\; s_3=|P_1-P_2|. \end{gather*}
Via CAS, obtain the following implicit equation for $X_1(\lambda)$
\begin{align*}
 f(x, x_1,\lambda)&=    \left( {a}^{2}{c}^{4}{\lambda}^{2}+2\,{a}^{2}{b}^{2}{c}^{2} \left( {a
}^{2}-2\, x_1^{2} \right) \lambda+{b}^{4}{a}^{6} \right) {x}^{2}\\
&+  x_1\left(  2(a^2b^2 - c^2\lambda) (a^4b^2 - 2b^2c^2 x_1^2 + a^2c^2\lambda) \right) x\\
& -a^2 x_1^2\left(3 a^4 b^4    - 4 b^4 c^2  x_1^2 - 4 a^4 b^2 \lambda + 6 a^2 b^2 c^2 \lambda - c^4 \lambda^2\right) =0,\\
g(y, y_1,\lambda)&=
{b}^{2} \left( {a}^{4}{b}^{4}-2\,{a}^{2}{c}^{2}\lambda\,{b}^{2}+4\,{a}
^{2}{c}^{2}\lambda\, y_1^{2}+{\lambda}^{2}{c}^{4} \right) {y}^{2
}\\
&+2\, y_1\, \left( {a}^{2}{b}^{2}+\lambda\,{c}^{2} \right) 
 \left( {a}^{2}{b}^{4}+2\,{a}^{2}{c}^{2} y_1^{2}-{b}^{2}{c}^{2}
\lambda \right) y\\
&- y_1^{2}{b}^{2} \left( 3\,{a}^{4}{b}^{4}+4\,{a
}^{4}{c}^{2} y_1^{2}-4\,{a}^{4}{b}^{2}\lambda-2\,{a}^{2}{c}^{2}
{b}^{2}\lambda  - {c^4\lambda}^{2}   \right) =0,\\
E( x_1, y_1)&=   b^2  x_1^2+ a^2 y_1^2 -a^2b^2=0.
\end{align*}

Consider the ellipse $\mathcal{I}_\lambda(x,y)=x^2/\alpha^2+y^2/\beta^2-1=0$, where
\begin{align*}
    \alpha&= {\frac {a}{{a}^{2}{b}^{2}-\lambda\,{c}^{2}} \left( 2\,{b}^{3}\sqrt {{b
}^{2}-\lambda}+{c}^{2} \left( {b}^{2}-\lambda \right) -{b}^{4}
 \right) },\\
 \beta&={\frac {b}{{a}^{2}{b}^{2}+\lambda\,{c}^{2}} \left( 2\,{a}^{3}\sqrt {{a
}^{2}-\lambda}-{c}^{2} \left( {a}^{2}-\lambda \right) -{a}^{4}
 \right) }.
\end{align*} 

It turns out that the parametrization of $X_1$ is a rational expression of the form
\[ \left[  \frac{ x_1(a_{20} x_1^2+a_{02} y_1^2)}{c_{20} x_1^2+c_{02} y_1^2},  \frac{ y_1(b_{20} x_1^2+b_{02} y_1^2)}{d_{20} x_1^2+d_{02} y_1^2}\right],\]
iff $\lambda=a^2b^2(2\delta -a^2 - b^2  )/c^4$. This is precisely the case when the triangle $P_1P_2P_3$ is a billiard orbit.

%\textcolor{red}{ron: check if factorization works}

The trace of $X_1(\lambda)$ is the ellipse $\mathcal{I}_\lambda(x,y)=0$ iff $P_1P_2P_3$ is a billiard orbit. This is obtained computing the affine curvature of $X_1(\lambda)$ and showing that it is constant only for the critical value $\lambda=a^2b^2(2\delta -a^2 - b^2  )/c^4$.
\end{proof}

\begin{proposition}
Over the conf-II family, the locus of $X_1$ is convex if $(a')^2>a^2-\lambda_o$, where $\lambda_o$ is the largest non-negative root of the following polynomial which is less than $a$ (in general there are three real roots):
\begin{align*} 
& c^8 \lambda_o^5 + b^2 (3 a^4 - 2 a^2 b^2 + 3 b^4) c^4 \lambda_o^4 +\\
&  2 a^2 b^4 (a^4 + 5 a^2 b^2 - 2 b^4) c^2  \lambda_o^3 + 2 a^4 b^6 (a^4 + b^4) \lambda_o^2 -\\
& b^8 a^6 (11 a^2 - 4 b^2) \lambda_o + 3 b^{10} a^8,
\end{align*}
where $c^2=a^2-b^2$.
\label{prop:confII-x1-convex}
\end{proposition}

%\begin{figure}
%    \centering
%    \includegraphics[width=\textwidth]{pics/0210_convex_x1.pdf}
%    \caption{\textbf{Top row:} non-convex locus of $X_1$ under conf-II for two aspect ratios of the outer ellipse (black). The the confocal caustic (brown) is slightly smaller than the threshold for convexity of the locus. \textbf{Bottom row:} the caustic now is at the threshold and the locus of $X_1$ has just turned convex. Note that the locus of $X_1$ for the aspect ratios shown passes very close to (though not exactly over) the foci.}
%    \label{fig:confII-x1-convex}
%\end{figure}

\subsection*{Studying the excenters}

As before, let $P_i'$ denote the excenter opposite to $P_i$.
Let $(a')^2=a^2-\lambda$ and $(b')^2=b^2-\lambda$ the semi-axes of the confocal ellipse. Referring to \cref{fig:confII-x1-excs}:
\begin{theorem}
Over the conf-II family, both $P_2'$ and $P_3'$ sweep a single ellipse $\E_e$, which is both concentric and axis-aligned with the Poncelet pair of ellipses. Its semi-axes $a_e,b_e$ are given by
\[ [a_e,b_e]  =\sqrt{k} \left[{\frac{{a}}{
 {a}^{2}{b}^{2} + c^2 \lambda
 }}, {\frac {{b} }{
{a}^{2}{b}^{2}  -c^2 \lambda}} \right], \]
where $k= \left(  \left( a+b \right) ^{2}\lambda+{a}^{2}{b}^{2}
 \right)  \left(  \left( a-b \right) ^{2}\lambda+{a}^{2}{b}^{2} \right)$.
% a_e^2& ={\frac {{a}^{2} \left(  \left( a+b \right) ^{2}\lambda+{a}^{2}{b}^{2} \right)  \left(  \left( a-b \right)^{2}\lambda+{a}^{2}{b}^{2} \right) }{ \left( {a}^{2}{b}^{2} + c^2 \lambda \right) ^{2}}}\\
%b_e^2&={\frac {{b}^{2} \left(  \left( a+b \right) ^{2}\lambda+{a}^{2}{b}^{2}\right)  \left(  \left( a-b \right) ^{2}\lambda+{a}^{2}{b}^{2} \right) }{ \left({a}^{2}{b}^{2}  -c^2 \lambda \right) ^{2}}}
\label{thm:confII-exc}
\end{theorem}

\begin{proof}
The excenter  $P_2'(\lambda)$ is parametrized by
\begin{gather*}
P_2'(\lambda)=  \frac{s_1P_1-s_2P_2+s_3P_3}{s_1-s_2+s_3},\\
s_1=|P_2-P_3|,\;\; s_2=|P_2-P_3|,\;\;s_3=|P_1-P_2|. \end{gather*}
Analogously, 
\begin{align*}
P_3'(\lambda)&= \frac{s_1 P_1+s_2 P_2-s_3 P_3}{s_1+s_2-s_3}. \end{align*}

Defining the ellipse by $\mathcal{E}_\lambda(x,y)=x^2/a_e^2+y^2/b_e^2-1=0$, with a CAS obtain that $\mathcal{E}_\lambda(P_2')=\mathcal{E}_\lambda(P_3')=0$.
\end{proof}

\begin{remark}
It can also be shown (via CAS) that over the conf-II family, the locus of $P_1'$ is a curve of degree 6.
\end{remark}

\subsection{Interesting excentral ellipses}

The two corollaries below are illustrated in \cref{fig:confII-n4-n6}.

\begin{corollary}
Let $a,b$ be the semi-axes of $\E$. Over the conf-II family, the envelope of $P_2 P_3$ is a point at the center of $\E$ for the following choice of $\E'$ semi-axes
\[ a'=\frac{a^2}{\sqrt{a^2+b^2}}, \;\; b'=\frac{b^2}{\sqrt{a^2+b^2}} \cdot \]
In such a case, the locus of $P_2'$ (and $P_3'$) is an ellipse with aspect ratio $b/a$.
\label{cor:confII-n4}
\end{corollary}

\begin{proof}
The expressions for $a',b'$ above are those of the confocal caustic $\E'$ which with $\E$ admits a 4-periodic family \cite[App. B.3]{garcia2020-self-intersected}. Indeed, opposite vertices of an even-N Poncelet family in a concentric, axis-parallel ellipse pair are antipodal with respect to the center, i.e., the major diagonals pass through the common center \cite{halbeisen2015}.

Under $a',b'$ in the proposition, $\lambda$ in \cref{thm:confII-exc} is given by
\[ \lambda=\frac{a^2b^2}{a^2+b^2} \cdot \]

Carrying out the simplifications, obtain $a_e/b_e = b/a$.
\end{proof}

\begin{corollary}
The locus of $P_2'$ and $P_3'$ is a circle concentric with the $\E,\E'$ 
when
\[ a'= \frac {a\sqrt {{a}(a+2\,b)}}{a+b},\;\;\;b'= {\frac {b\sqrt {b \left( 2\,a+b \right) }}{a+b}},\;\;\;\lambda=\left(\frac{a b}{a+b}\right)^2 \cdot\]
\label{cor:confII-n6}
\end{corollary}

Note that in such a case, $a',b'$ are precisely the semi-axes of the elliptic billiard $N=6$ family \cite[Appendix B.5] {garcia2020-self-intersected}. 

\begin{figure}
    \centering
    \includegraphics[width=\textwidth]{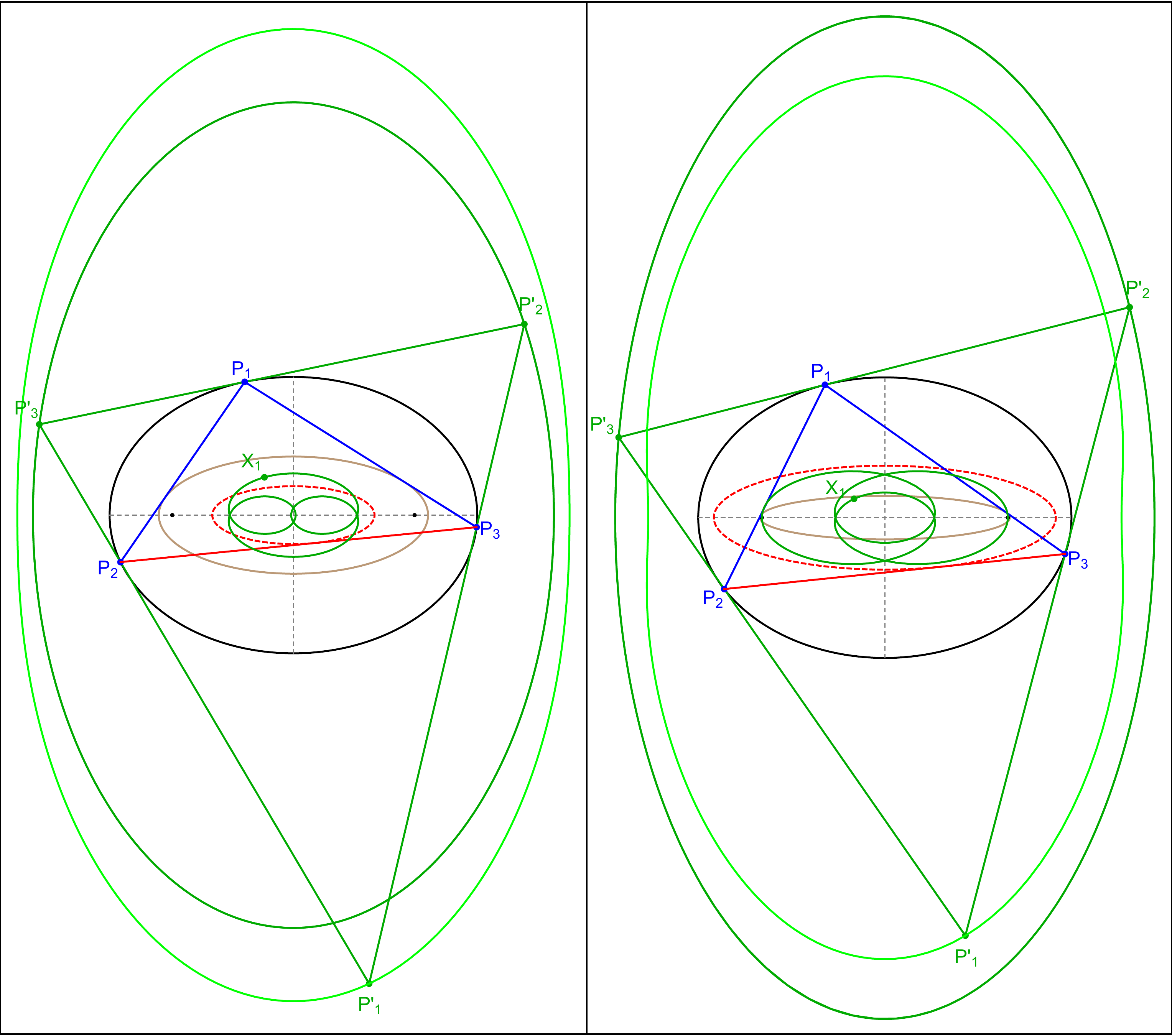}
    \caption{The non-elliptic locus of the incenter $X_1$ and that of the vertices $P_i'$ of the excentral triangle (dark green) over a conf-II family with a (i) larger (left), and (ii) smaller (right) caustic than the $N=3$, conf-I one. Remarkably, excenters $P_2'$ and $P_3'$ survive the asymmetry sweeping a single ellipse (dark green), while $P_1'$ sweeps a non-conic (light green) interior (resp. exterior) to the former one. \href{https://youtu.be/AisIrfn4IGg}{Video}}
    \label{fig:confII-x1-excs}
\end{figure}

\begin{figure}
    \centering
    \includegraphics[width=\textwidth]{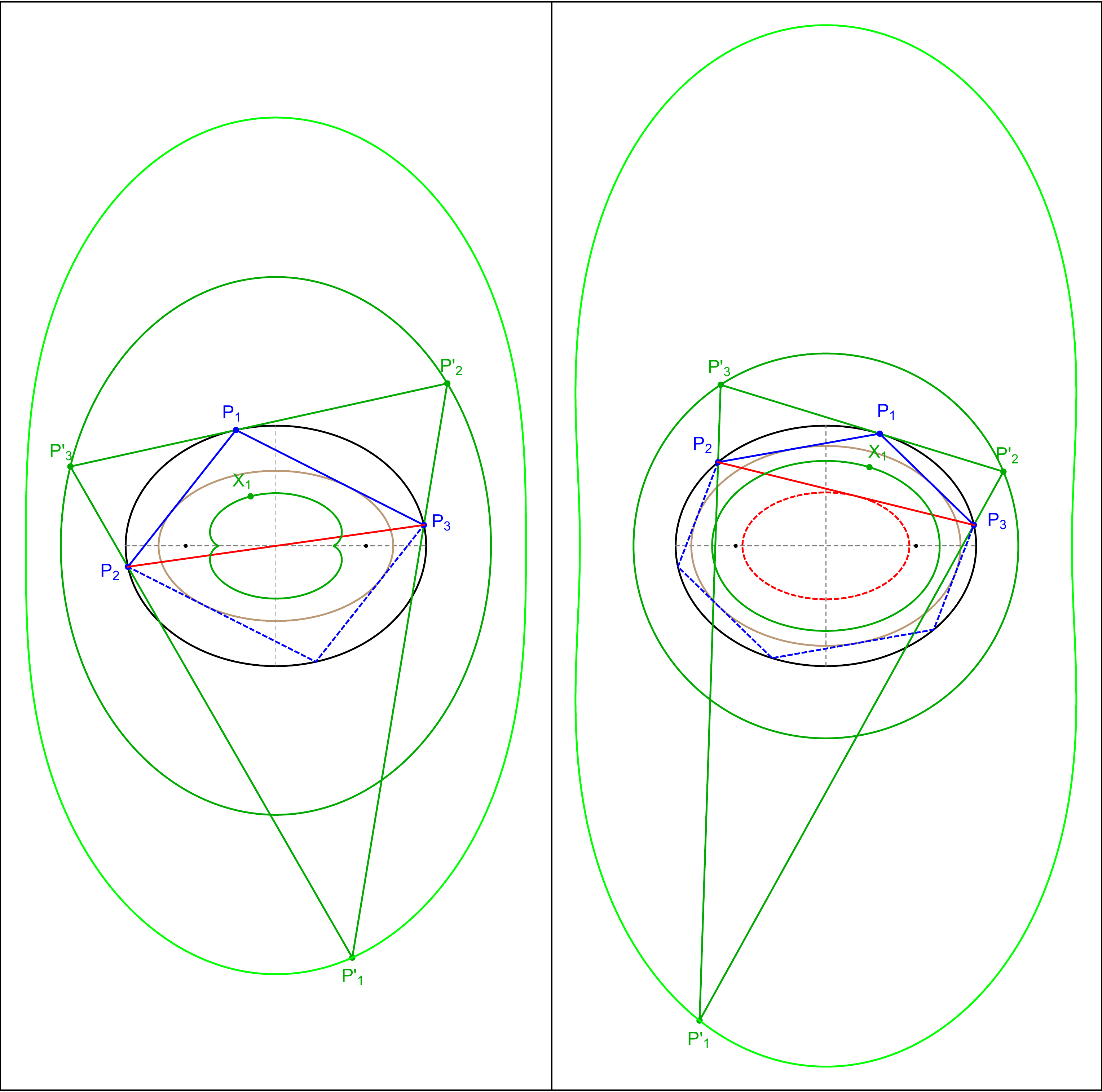}
    \caption{\textbf{Left:} the conf-II caustic $\E'$ (brown) is chosen such that the envelope of $P_2 P_3$ (red) is the center of the system. Note that $\E'$ is the caustic required for the $N=4$ trajectory (dashed blue). In such a case, the axis ratio of the elliptic locus of excenters $P_2',P_3'$ (dark green) is the reciprocal of that of $\E$ (black). \textbf{Right:} The caustic (brown) is now chosen such that if one were to complete the Poncelet iteration about $\E'$ one would obtain an $N=6$ trajectory (dashed blue). In such a case, the locus of $P_2',P_3'$ is a circle! Note that in both cases the locus of $P_1'$ (light green) and that of the incenter $X_1$ (dark green interior to $\E'$) are non-conics. \href{https://youtu.be/wB9bVkY9rqU}{Video}}
    \label{fig:confII-n4-n6}
\end{figure}

\subsection{The barycenter locus}

That the locus of the barycenter $X_2$ (and many other triangle centers) is an ellipse over the conf-I family has been established \cite{garcia2020-new-properties}. For the general conf-II case, it is not a conic. As shown in \cref{fig:confII-x2-n4} (top left), it is a conic if the caustic is such that the envelope of $P_2 P_3$ is the center point. Indeed:

\begin{proposition}
Over a family of triangles $P_1 P_2 P_3$ inscribed in an outer ellipse $\E$, where sides $P_1 P_2$ and $P_1 P_3$ are tangent to an inner, concentric, axis-parallel ellipse $\E'$, and the third side $P_2 P_3$ envelops the center point, the locus of $X_2$ will be homothetic to $\E$, at $1/3$ scale.
\end{proposition}

The proof was kindly contributed by Mark Helman \cite{helman2021-private}.

\begin{proof}
Let $P_4$ be the reflection of $P_1$ about the center O. Since an even-N Poncelet family in a concentric, axis-parallel ellipse pair has (i) parallel opposite sides, and (ii) diagonals which pass through $O$ \cite{halbeisen2015}, $P_1 P_2 P_3 P_4$ will be a parallelogram and $P_1 P_4$ will be such a diagonal. So $O$ is the midpoint of $P_2 P_3$, meaning $X_2$ of $P_1 P_2 P_3$ will be two-thirds of the way between $P_1$ and $O$. This sends the outer ellipse to a copy of itself scaled by $1/3$.
\end{proof}

\begin{figure}
    \centering
    \includegraphics[width=\textwidth]{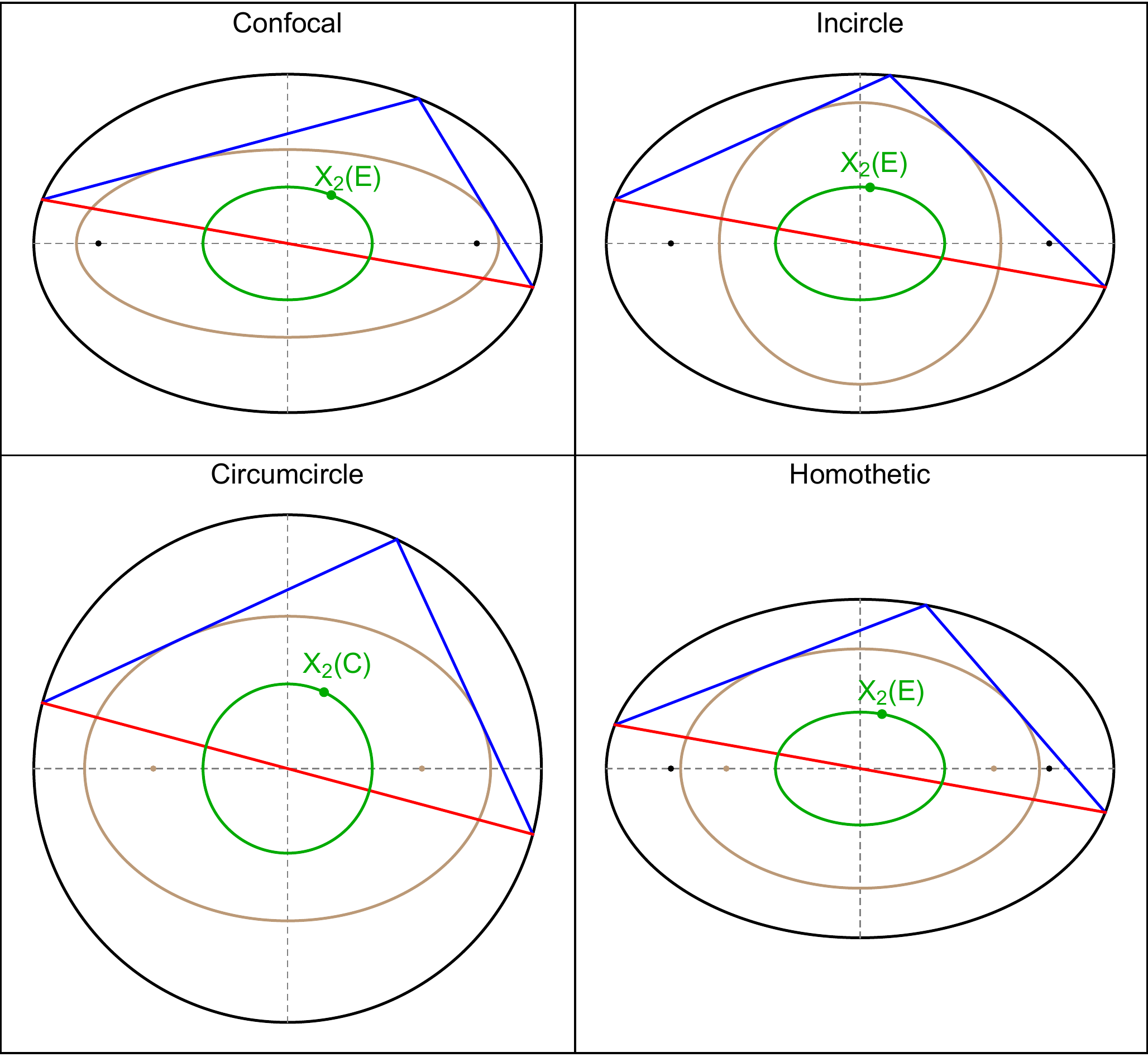}
    \caption{In the conf-II family (top left), where the caustic is selected such that the envelope of the free side (red) is the origin, the locus of $X_2$ is an ellipse. Indeed, the locus of $X_2$ is an ellipse for any triangle families inscribed in an outer ellipse and with two sides tangent to an inner concentric, axis-parallel one, such that the third side's envelope is the origin. \href{https://youtu.be/6yNod1LFVrY}{Video}}
    \label{fig:confII-x2-n4}
\end{figure}

The elliptic locus of $X_2$ is shown over other concentric, axis-parallel ``half $N=4$'' families in \cref{fig:confII-x2-n4}.

\section{Families with Three Caustics}
\label{sec:three-caustics}
\subsection{Adding another caustic to bic-II}

We now consider a natural extension to the bic-II family, namely, a family of triangles $P_1 P_2 P_3$ inscribed in an outer circle $\C$ where $P_1 P_2$ (resp. $P_1 P_3$) is tangent to a first circle $\C'$ (resp. second circle $\C''$). In particular, $\C''$ is picked to be in the pencil of $\C,\C'$. Note that by \cref{lem:pml}, the third side $P_2 P_3$ will be tangent to a third, in-pencil circle $\C'''$. We term this the ``bic-III'' family.

For the purposes of the next set of results, assume all four in-pencil circles in a bic-III family are distinct. Referring to \cref{fig:bicIII}, experimental evidence suggests:

\begin{conjecture}
Over the bic-III family, the locus of $X_1$ is a convex curve.
\end{conjecture}

\begin{conjecture}
Over the bic-III family, the locus of a triangle center is never a conic.
\end{conjecture}

\begin{figure}
    \centering
    \includegraphics[width=\textwidth]{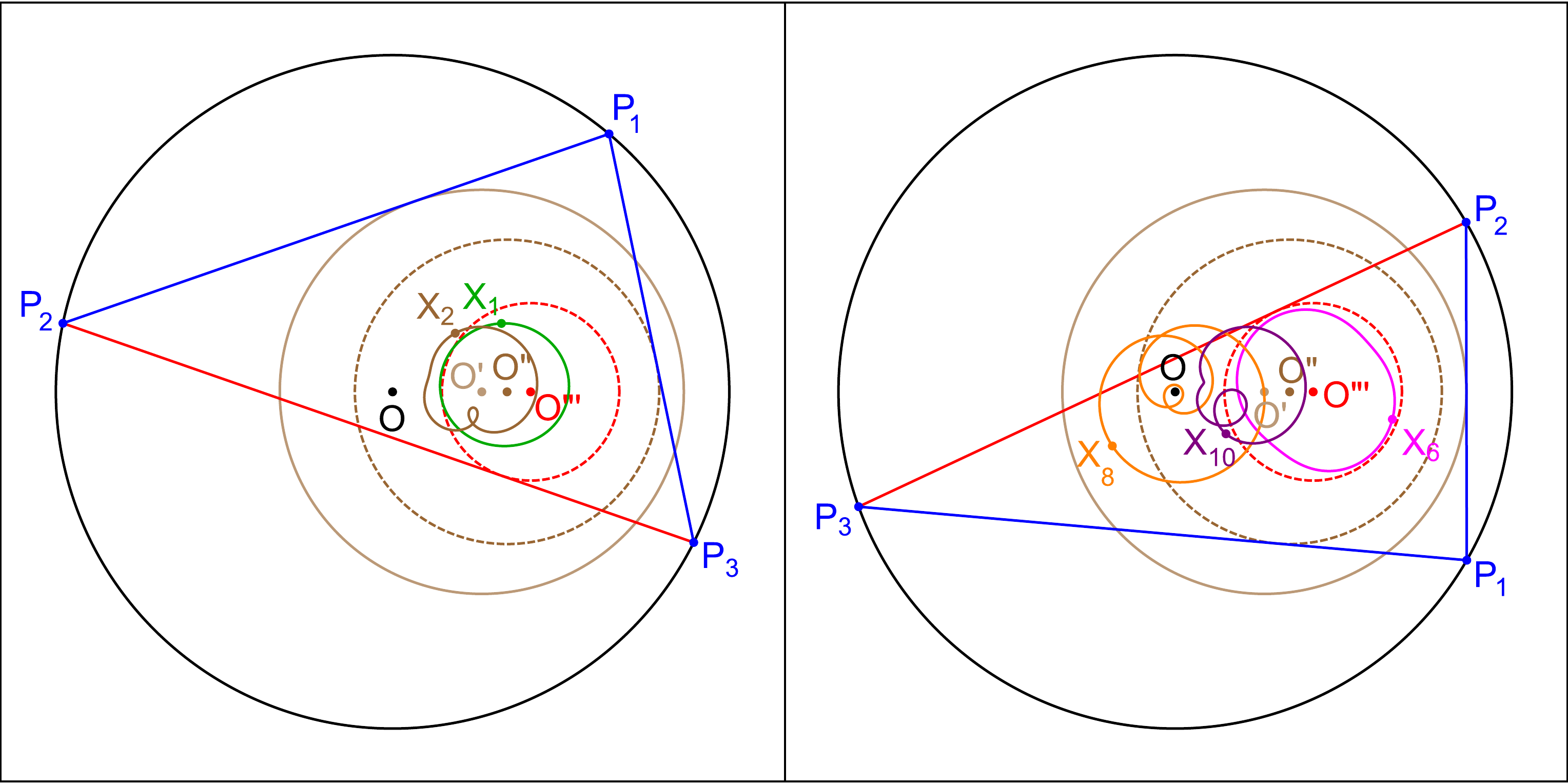}
    \caption{Non-conic loci of triangle centers over bic-III triangles, vertices lie on an outer circle (black) while each side is tangent to a distinct in-pencil circle (brown, dashed brown, and dashed red). \textbf{Left:} loci of the incenter $X_1$ (green) and of the barycenter $X_2$ (brown). Experimentally, the locus of $X_1$ is always convex. \textbf{Right:} loci of the symmedian point $X_6$ (pink), Nagel point $X_8$ (orange), and Spieker center $X_{10}$ (purple). Though in the current choice the locus of $X_6$ looks convex, in general it is not. \href{https://youtu.be/cvB0A7LmlZc}{Video}}
    \label{fig:bicIII}
\end{figure}

Referring to \cref{fig:bicIII-excs}, we also find that:

\begin{conjecture}
Over the bic-III family, the excenters sweep three distinct non-conic curves.
\end{conjecture}

\begin{figure}
    \centering
    \includegraphics[width=\textwidth]{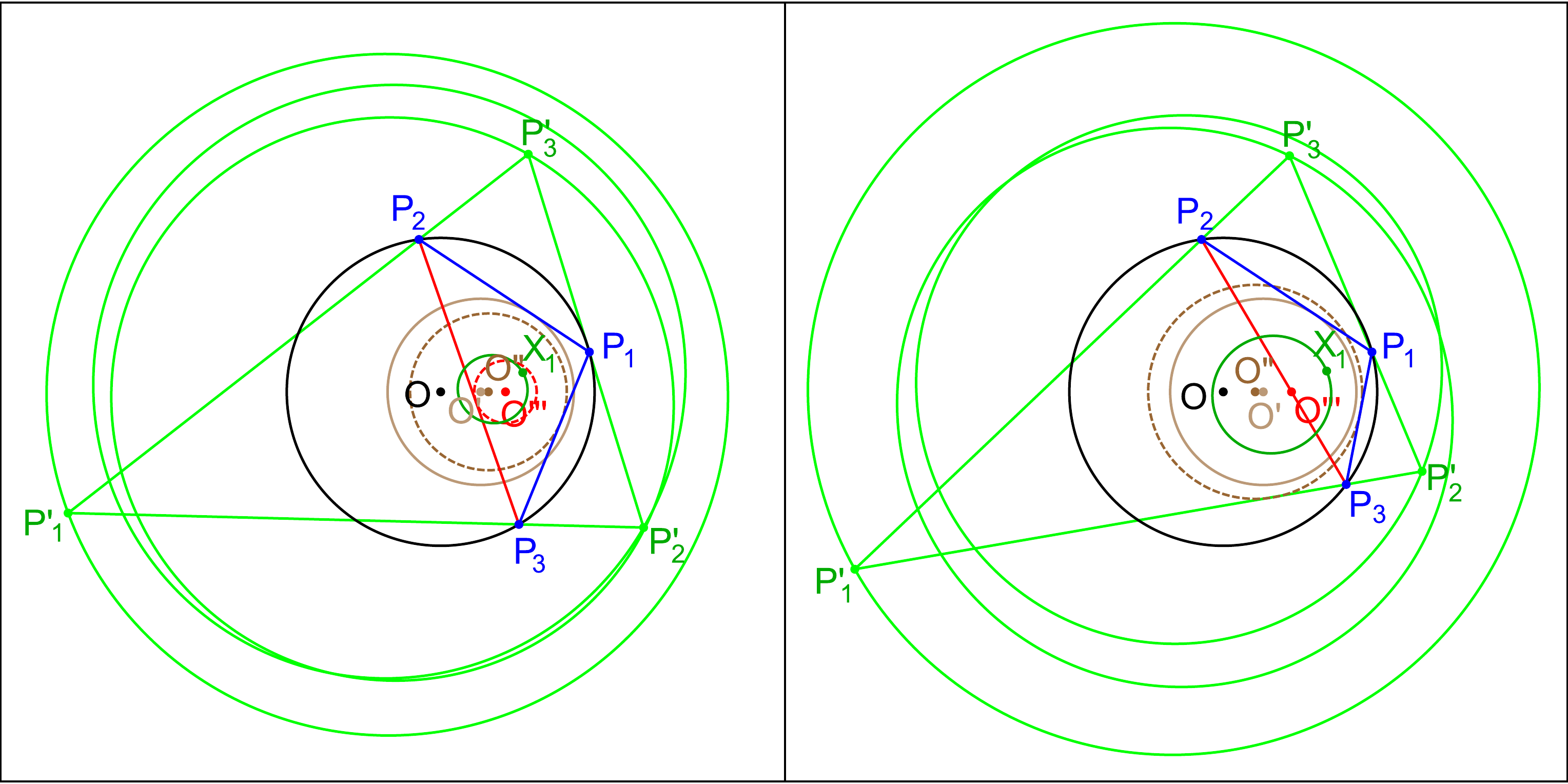}
    \caption{\textbf{Left:} bic-III triangles $P_1 P_2 P_3$, where the envelope of $P_2 P_3$ (dashed red) is a circle with non-zero radius. The loci (light green) of the three excenters $P_i'$ are all non-conics, as is that of the incenter $X_1$ (dark green). \textbf{Right:} two caustics (brown and dashed brown) are selected so that the envelope of $P_2 P_3$ is the (internal) limiting point of the pencil. Even in this case, the loci of all excenters are non-conics, though that of $P_1'$ is numerically very close to a circle. \href{https://youtu.be/A_-U2VvM5kY}{Video}}
    \label{fig:bicIII-excs}
\end{figure}

\subsection{Adding another in-pencil caustic to conf-II}

We call ``conf-III'' a family of triangles inscribed in an ellipse $\E$, with a first side tangent to a confocal caustic $\E'$, and a second one tangent to a distinct ellipse $\E''$ in the pencil spanned by $\E,\E'$ (it is therefore non-confocal). Again, \cref{lem:pml} ensures that the third side will envelop a third ellipse in the same pencil. As in the bic-III family, no triangle centers produce conic loci. As shown in \cref{fig:conf-III}, these can be quite convoluted.

\begin{figure}
    \centering
    \includegraphics[width=.8\textwidth]{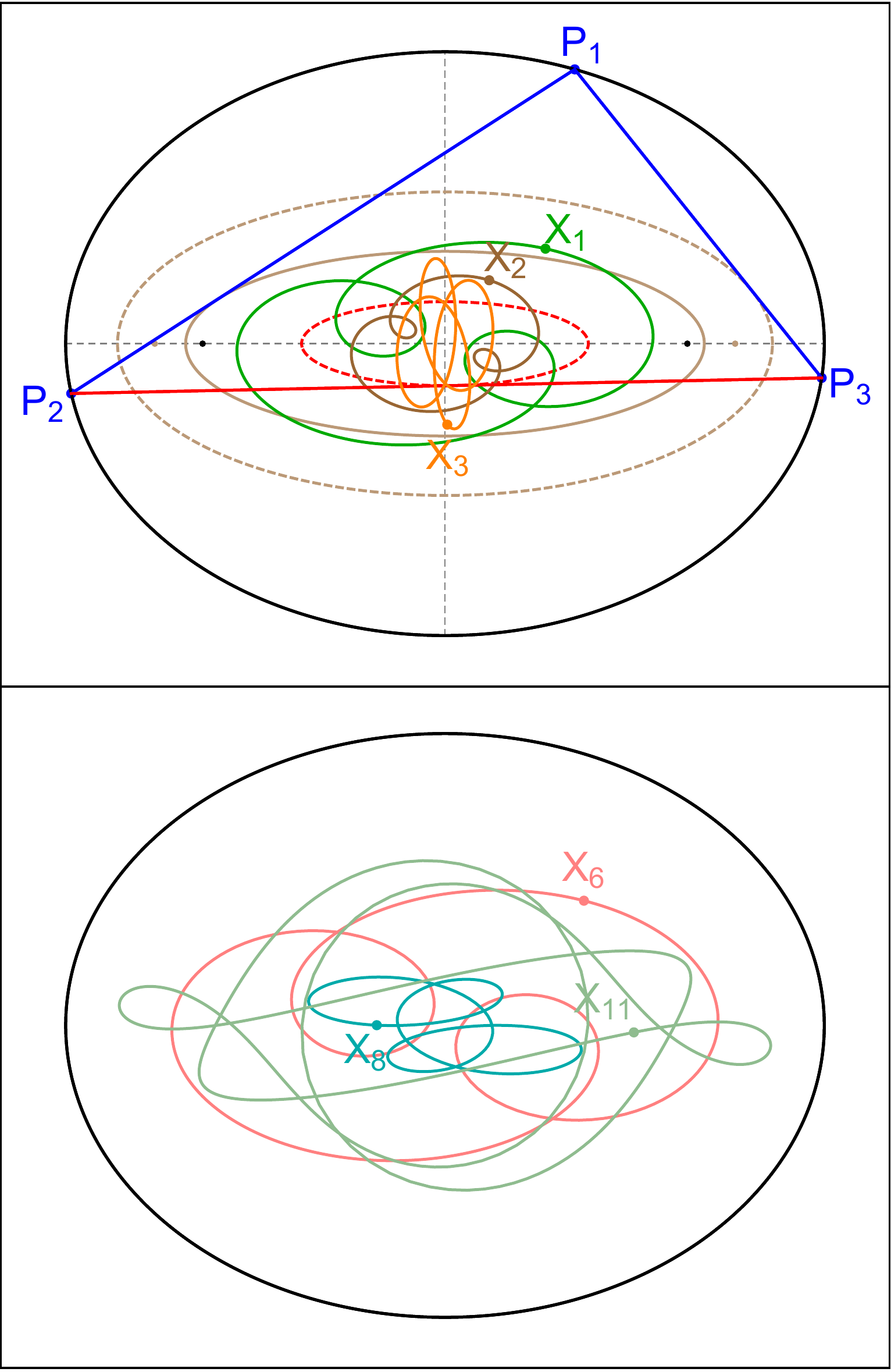}
    \caption{\textbf{Left:} A conf-III family inscribed in an outer ellipse (black) and with each side tangent to a separate in-pencil caustic (brown, dashed brown, and red). Only the first of these (brown) is confocal. The non-conic, convoluted loci of $X_k$, $k=1,2,3$ are shown (green, brown, green). \textbf{Right:} the same setup, now showing the whirly loci of $X_k$, $k=6,8,11$. For simplicity, the three inner caustics have been omitted.}
    \label{fig:conf-III}
\end{figure}

\subsection{Choosing one of four triangles}

Referring to \cref{fig:four-tris-x1x2}, there are four ways in which a bic-III triangle can be chosen, namely, two tangents of $P_1 P_2$ over $\C'$, and two tangents of $P_1 P_3$ over $\C''$. As it turns out, these only produce two distinct envelopes for $P_2 P_3$. For each choice, a triangle center such as the incenter or barycenter will sweep a distinct locus.

\begin{figure}
    \centering
    \includegraphics[width=\textwidth]{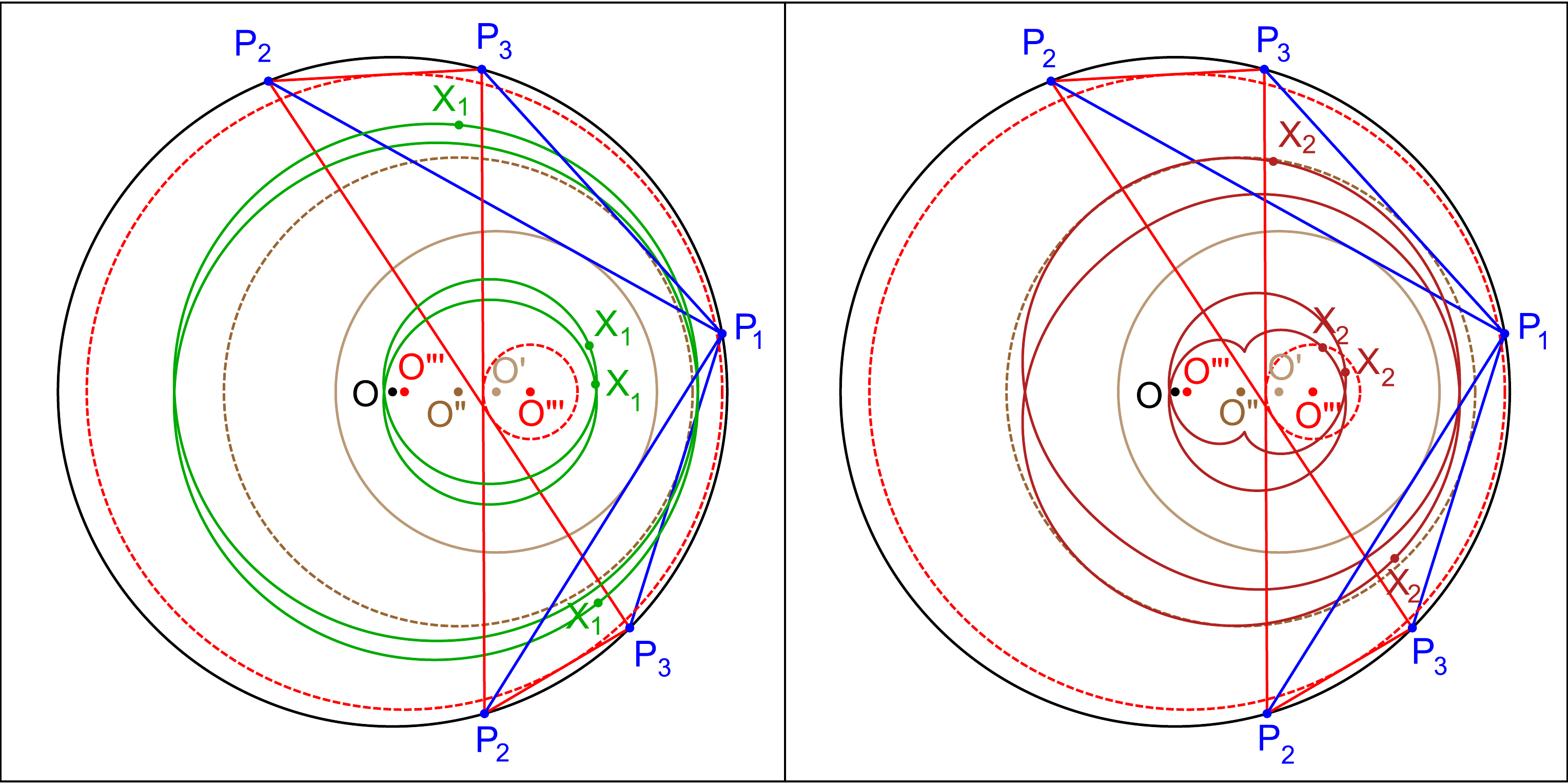}
    \caption{The four possible choices of $P_1 P_2$ and $P_1 P_3$ over the bic-III family, and the locus of the incenter (green, left), and the barycenter (brown, right), traced out under each choice. Note: only two distinct envelopes (dashed red) of $P_2 P_3$ are produced. \href{https://youtu.be/E1Rcu38MePQ}{Video 1}, \href{https://youtu.be/u1_uANWDNr8}{Video 2}}
    \label{fig:four-tris-x1x2}
\end{figure}

\section{Summary and Videos}
\label{sec:videos}

\begin{table}
\setlength{\tabcolsep}{3pt}
\begin{tabular}{|r|l|c|cccccc|}
\hline
family & note, aka. & {\scriptsize caustics} & $X_1$ & $X_2$ & $X_3$ & $P_1'$ & $P_2'$ & $P_3'$ \\
 \hline
bic-I & poristic, bicentric & 1 & P & C & P & C & C & C \\
bic-II & tricentric & 2 & C & 6 & P & C & 6 & 6 \\
bic-III & quadricentric & 3 & N$^\dagger$ & N & P & N$^\dagger$ & N$^\dagger$ & N$^\dagger$ \\
\hline
conf-I & confocal, elliptic billiard$^\ddagger$ & 1 & E & E & E & E & E & E \\
conf-II & caustics: 1 conf., 1 in-pencil & 2 & N$^\dagger$ & N & N & 6 & E & E \\
conf-III & 1 conf., 2 in-pencil & 3 & N & N & N & N & N & N\\
%conf-III' & 2 confs., 1 not-in-pencil & 3 & N & N & N & N & N & N\\
\hline
\end{tabular}
\caption{Loci types for $X_1$, $X_2$, and the excenters for each family mentioned in the article. Symbols $P$, $C$, $E$, and $6$ stand for point, circle, ellipse, and sextic curves, respectively. An ``N'' denotes a non-conic whose degree hasn't yet been determined. $^\dagger$ A non-conic locus is conjectured to be convex. $^\ddagger$ The mittenpunkt $X_9$ is stationary over conf-I \cite{garcia2020-new-properties}.}
\label{tab:summary}
\end{table}

Animations illustrating some of the above phenomena are listed in Table~\ref{tab:playlist}.

\begin{table}
\small
\begin{tabular}{|c|c|l|l|}
\hline
id & Fam. & Title & \textbf{youtu.be/<.>}\\
\hline
01 & n/a & Poncelet's porism with 2 ellipses, $N=7$ & \href{https://youtu.be/kzxf7ZgJ5Hw}{\texttt{kzxf7ZgJ5Hw}}\\
02 & n/a & Poncelet's closure theorem (PCT) & \href{https://youtu.be/L5A_S4VQLiw}{\texttt{L5A\_S4VQLiw}}\\
03 & all & $N=3$ families with multiple caustics & \href{https://youtu.be/8HXgkuY-nFQ}{\texttt{8HXgkuY-nFQ}} \\
\hline
04 & bic-II & Bicentric family with 1,2 caustics &
\href{https://youtu.be/OM7uilfdGgk}{\texttt{OM7uilfdGgk}}\\
05 & bic-II & Circular loci of $X_k,k=$1,40,165 & \href{https://youtu.be/qJGhf798E-s}{\texttt{qJGhf798E-s}}\\
06 & bic-II & Loci of $X_1$ and $X_2$ over varying caustic radius & \href{https://youtu.be/3dnsWPlAmxE}{\texttt{3dnsWPlAmxE}}\\
07 & bic-II & Loci of $X_k,k=2,4,5$ over varying caustic radius & \href{https://youtu.be/6Fqp6Z1Q-0A}{\texttt{6Fqp6Z1Q-0A}}\\
08 & bic-II & Loci of the excenters (a circle and 2 non-conics) & \href{https://youtu.be/qdqIuT-Qk6k}{\texttt{qdqIuT-Qk6k}}\\
\hline
09 & conf-II & Confocal family with 1,2 caustics & \href{https://youtu.be/C14TL430UBc}{\texttt{C14TL430UBc}}\\
10 & conf-II & Locus of $X_1$ under sweep of confocal caustics & \href{https://youtu.be/kCY6KHFDV2M}{\texttt{kCY6KHFDV2M}}\\
11 & conf-II & Locus of $X_1$ and excenters & \href{https://youtu.be/AisIrfn4IGg}{\texttt{kCY6KHFDV2M}}\\
12 & conf-II & Loci of the excenters with the $N=4,6$ caustics & \href{https://youtu.be/wB9bVkY9rqU}{\texttt{wB9bVkY9rqU}}\\
\hline
13 & bic-III & Non-conic loci of $X_k,k=$1, 2, 6, 8, 10 & \href{https://youtu.be/cvB0A7LmlZc}{\texttt{cvB0A7LmlZc}}\\
14 & bic-III & Non-conic loci of incenter and excenters & \href{https://youtu.be/A_-U2VvM5kY}{\texttt{A\_-U2VvM5kY}}\\
15 & bic-III & Loci of $X_1,X_2$ under 4 possible triangle choices & 
\href{https://youtu.be/E1Rcu38MePQ}{\texttt{E1Rcu38MePQ}}\\
16 & bic-III & Loci of $X_6,X_8$ under 4 possible triangle choices & 
\href{https://youtu.be/u1_uANWDNr8}{\texttt{u1\_uANWDNr8}}\\
\hline
17 & n/a & Locus of $X_2$ over ``half $N=4$'' families & \href{https://youtu.be/6yNod1LFVrY}{\texttt{6yNod1LFVrY}}\\
18 & n/a & Loci of $X_{11}$ and $X_{59}$ over varying circular caustic & \href{https://youtu.be/2VqgB6KvP2g}{\texttt{2VqgB6KvP2g}}\\
\hline
\end{tabular}
\caption{Table of videos. The last column is clickable and provides the YouTube code.}
\label{tab:playlist}
\end{table}

\section*{Acknowledgements}
\noindent We would like to thank Alexey Zaslavsky for the proof of Theorem 2 and Mark Helman for proposing the variant we use here. Also invaluable has been Arseniy Akopyan's generous availability to review and answer dozens of questions about many of the results herein. The first author is fellow of CNPq and coordinator of Project PRONEX/ CNPq/ FAPEG 2017 10 26 7000 508.

\appendix

\section{Symbol Table}
\label{app:symbols}
\cref{tab:symbols} lists most symbols used herein.

{\small
\begin{table}
\begin{tabular}{|c|l|}
\hline
symbol & meaning \\
\hline
$\C,\C'$ & outer and inner circle in a poristic pair \\
$R,r,d$ & radii of $\C$ and $\C'$ and distance between their centers \\
$\C''$ & circular envelope of side $P_2 P_3$ in the pencil of $\C,\C'$\\
$\E,\E'$ & outer and inner ellipses in confocal pair \\
$\E''$ & non-confocal envelope of side $P_2 P_3$ in the pencil of $\E,\E'$\\
$O,O'$ & centers of either $\C,\C'$ or $\E,\E'$ \\
$a,b$ & major and minor semi-axes of $\E$ \\
$c$ & half the focal length of $\E$ \\
$a',b'$ & major and minor semi-axes or $\E'$ \\
\hline
$X_1$ & Incenter \\
$X_2$ & Barycenter \\
$X_3$ & Circumcenter \\
$X_4$ & Orthocenter \\
$X_5$ & Euler's circle center \\
$X_6$ & Symmedian point \\
$X_9$ & Mittenpunkt \\
$X_{40}$ & Bevan point \\
\hline
\end{tabular}
\caption{Symbols used in the article.}
\label{tab:symbols}
\end{table}
}

\section{Family Parametrizations}
\label{app:vtx-param}
\subsection{Bic-I}
This family is also known as poristic or bicentric triangles. To obtain its vertices, use the parametrization for bic-II (see below), setting $d=\sqrt{R(R-2r)}$.

\subsection{Bic-II}\label{subsec:Bic-II}
Let $P_i=[x_i,y_i]$ denote the vertices of a bic-II triangle, $i=1,2,3$. Then:
{\small
\begin{align*}
  x_2&=\frac{(2 r y_1 (R^2 - d^2)  \Delta+   (2 d R^2   - ( R^2   + d^2) x_1) (\Delta^2 -   r^2))}{(R^2 + d^2 - 2 d x_1)^2}, \\
  y_2&=\frac{  (4R^2 d - 2 (R^2  + d^2) x_1) r \Delta - y_1 (R^2 - d^2)  (\Delta^2- r^2)}{ (R^2 + d^2 - 2 d x_1)^2},\\
  x_3&=x_2-\frac{4(R^2 - d^2) y_1 r\Delta}{(R^2 + d^2 - 2dx_1)^2},\;\;
  y_3=y_2- \frac{4(2R^2d - (R^2+d^2) x_1) r\Delta}{(R^2 + d^2 - 2dx_1)^2},
%x_2&=\frac{2 (R^2 - d^2)\Delta y_1 + 4d(R^2 + d^2 - r^2)x_1^2 + (-R^4 - 6R^2d^2 + 2R^2r^2 - d^4 + 2d^2r^2)x_1 + 2d(R^2 + d^2 - 2r^2)y_1^2}{(R^2 + d^2 - 2 d x_1)^2 }\\
% y_2&=\frac{(-2 R^2 - 2 d^2)\Delta  x_1 + 4 R^2 d\Delta + 2 (R^2 - d^2)  d  x_1  y_1 -  y_1 (R^2 - d^2)   (R^2 + d^2 - 2 r^2)}{(R^2 + d^2 - 2 d x_1)^2} \\
% x_3&=\frac{4 d (R^2 + d^2 - r^2)  x_1^2 + (-R^4 - 6 R^2 d^2 + 2 R^2 r^2 - d^4 + 2 d^2 r^2)  x_1 + 2 d (R^2 + d^2 - 2 r^2)  y_1^2 - 2 \Delta (R^2 - d^2)  y_1}{(R^2 + d^2 - 2 d x_1)^2}\\
% y_3&=\frac{2 d (R^2 - d^2)     x_1  y_1 +2 (  R^2 +   d^2) \Delta  x_1 -  y_1 (R^2 - d^2)  (R^2 + d^2 - 2 r^2) - 4 R^2 d \Delta}{(R^2 + d^2 - 2 d x_1)^2}
\end{align*}}
where $\Delta=\sqrt{R^2 + d^2 - 2 d x_1 - r^2}$.
%$$ x_2=\frac{(2 r y_1 (R^2 - d^2)  \Delta+   (2 d R^2   - ( R^2   + d^2) x_1) (\Delta^2 -   r^2))}{(R^2 + d^2 - 2 d x_1)^2} $$

%$$y_2=\frac{  (4R^2 d - 2 (R^2  + d^2) x_1) r \Delta - y_1 (R^2 - d^2)  (\Delta^2- r^2)}{ (R^2 + d^2 - 2 d x_1)^2} $$

%$$x_3=x_2-\frac{4(R^2 - d^2) y_1 r\Delta}{(R^2 + d^2 - 2dx_1)^2}$$

%$$y_3=y_2- \frac{4(2R^2d - (R^2+d^2) x_1) r\Delta}{(R^2 + d^2 - 2dx_1)^2}$$
\subsection{Bic-III}
Consider the pencil of circles 
$ \C_u=(1-u)\C_1+u C_2$
where 
\[ \C_1:\; (x-d)^2+y^2=r^2,\;\;\C_2 :\; x^2+y^2=R^2.\]
The center of $\C_u$ is $(d(u),0)=(d(1-u),0)$. The   radius $r(u)$ of $\C_u$ is given by
\[r(u)=\sqrt{d^2u^2 + (R^2 - d^2 - r^2)u + r^2}.\]

Let $P_i=[x_i,y_i]$ denote the vertices of a bic-III triangle, $i=1,2,3$, where $P_1 P_2$ is tangent to $\C_1$ and $P_2 P_3(u)$ tangent to $\C_u$. Then {\small
\begin{align*}
  x_2(u)&=\frac{(2 r(u) y_1 (R^2 - d(u)^2)  \Delta(u)+   (2 d(u) R^2   - ( R^2   + d(u)^2) x_1) (\Delta(u)^2 -   r(u)^2))}{(R^2 + d(u)^2 - 2 d(u) x_1)^2}, \\
  y_2(u)&=\frac{  (4R^2 d(u) - 2 (R(u)^2  + d(u)^2) x_1) r(u) \Delta(u) - y_1 (R(u)^2 - d(u)^2)  (\Delta(u)^2- r(u)^2)}{ (R^2 + d(u)^2 - 2 d(u) x_1)^2},\\
  x_3(u)&=x_2(u)-\frac{4(R^2 - d(u)^2) y_1 r(u)\Delta(u)}{(R^2 + d(u)^2 - 2d(u) x_1)^2},\\
  y_3(u)&=y_2(u)- \frac{4(2R^2d(u) - (R^2+d(u)^2) x_1) r(u)\Delta(u)}{(R^2 + d(u)^2 - 2d(u) x_1)^2}.
  \end{align*}
}
where $\Delta(u)=\sqrt{R^2 + d(u)^2 - 2 d(u) x_1 - r(u)^2}$.
\subsection{Conf-I}

This family is also known as the ``elliptic billiard''. Its vertices can be obtained from the conf-II family (see below), setting $a',b'$ as in \cref{eqn:confocal}.

\subsection{Conf-II}
\label{subsec:Conf-II}
{\small
\begin{align*}
P_1&=[x_1,y_1],\\
P_2&=\left[\frac {2\,  a\alpha_3\,y_1 \Delta- \,\alpha_1\,\alpha_2 x_1} {W}
,{\frac {-2\,{b}^{2} \,\alpha_2\,x_1\Delta-a\alpha_1\,\alpha_3\,y_1}{
a W}}\right],\\
P_3&=\left[ \frac{-2\, a\alpha_3\,y_1\Delta - \,\alpha_1\,\alpha_2 x_1}{W},
{\frac {2\,{b}^{2} \,\alpha_2 x_1\Delta  -a  \alpha_1\,\alpha_3\,y_1}{a{
  W}}}\right],\\
W&= \frac{\alpha_2^{2} x_1^{2}}{a^2}+  \frac{\alpha_3^{2} y_1^2}{b^2},\;\;\Delta^2=\left( {a}^{2}b'^{2}-a'^{2}b'^{2} \right) x_1^{2}+ \left( {a}^{2}a'^{2}-{\frac {{a}^{2}a'^{2} b'^{2}}{ b^{2}}} \right) y_1^{2},\\
\alpha_1& =a^2(b^2 -  b'^2) - a'^2 b^2,\;\;\alpha_2=(a^2-a'^2)b^2 + a^2b_c^2,\;\; \alpha_3=a^2( b^2-  b'^2)+a'^2 b^2.
\end{align*}}
%\subsection{Conf-III}
%We omit this due to the complexity.

\section{Locus of barycenter over bic-II}
\label{app:locus-x2}

Consider the bic-II family, parametrized by $P_1=[x_1,y_1]=R[\cos{t},\sin{t}]$. The locus of the barycenter $X_2$ of $P_1 P_2 P_3$ is given by
{\small
\begin{align*}
    X_2&=\left[
\frac {-4\,{d}^{2}  x_1\, y_1^{2} - \left(  {d}^{4}+ \left( 6
\,{R}^{2}-4\,{r}^{2} \right) {d}^{2}+{R}^{4}-4\,{R}^{2}{r}^{2}
 \right) x_1+4\,{R}^{2}d ({d}^{2}+    {R}^{2}-2 r^{2}) }{3\, \left( {R}^{2}+{d}^{2}-2\,d x_1 \right) ^{2}},\right.\\
&\left.-\frac { \left( -4\,{d}^{2}x_1^{2}+8\,{d}^{3}x_1-3\,{d}^{
4}+ \left( -2\,{R}^{2}+4\,{r}^{2} \right) {d}^{2}+{R}^{4}-4\,{R}^{2}{r
}^{2} \right) y_1}{3\, \left( {R}^{2}+{d}^{2}-2\,d x_1
 \right) ^{2}}\right]\cdot
\end{align*}
}
Using the method of resultants, eliminate $x_1,y_1$ from the system $X_2-[x,y]=0,\; x_1^2+y_1^2-R^2=0$. Then $X_2$ will be given by an implicit polynomial equation $f(x,y)=0$ of degree 6.
Explicitly:
{\small
\[
\begin{array}{l}
729(x^2+y^2)^3-972dx(x^2+y^2)((4R^2+5d^2-4r^2)y^2+(4R^2-3d^2-4r^2)x^2)-\\
-81(x^2+y^2)((R^6-86R^4d^2-8R^4r^2+121R^2d^4+128R^2d^2r^2+16R^2r^4-36d^6-\\
-72d^4r^2-96d^2r^4)x^2+(R^6-42R^4d^2-8R^4r^2-63R^2d^4+128R^2d^2r^2+16R^2r^4-\\
-4d^6+24d^4r^2-32d^2r^4)y^2)+ 108dx((3R^8-45R^6d^2-28R^6r^2+81R^4d^4+44R^4d^2r^2+\\
+80R^4r^4-39R^2d^6+20R^2d^4r^2-112R^2d^2r^4-64R^2r^6-36d^6r^2+64d^2r^6)x^2+\\
+(3R^8-45R^6d^2-28R^6r^2+81R^4d^4+76R^4d^2r^2+80R^4r^4-39R^2d^6-44R^2d^4r^2+\\
+16R^2d^2r^4-64R^2r^6-4d^6r^2+64d^2r^6)y^2)-36d^2((9R^{10}-36R^8d^2-90R^8r^2+\\
+54R^6d^4-18R^6d^2r^2+312R^6r^4-36R^4d^6+306R^4d^4r^2-372R^4d^2r^4-480R^4r^6+\\
+9R^2d^8-198R^2d^6r^2+96R^2d^4r^4+256R^2d^2r^6+384R^2r^8-36d^6r^4+96d^4r^6-64d^2r^8)x^2+\\
(9R^{10}-36R^8d^2-102R^8r^2+54R^6d^4+126R^6d^2r^2+392R^6r^4-36R^4d^6+54R^4d^4r^2-\\
-116R^4d^2r^4-544R^4r^6+9R^2d^8-78R^2d^6r^2+160R^2d^4r^4+128R^2r^8-4d^6r^4+\\
+32d^4r^6-64d^2r^8)y^2)+48R^2d^3r^2(3R^2-3d^2+4r^2)(3R^6-9R^4d^2-28R^4r^2+\\
+9R^2d^4-4R^2d^2r^2+80R^2r^4-3d^6+32d^4r^2-32d^2r^4-64r^6)x-\\
-16R^2d^4r^4((R+d)^2-4r^2)((R-d)^2-4r^2)(3R^2-3d^2+4r^2)^2=0.
\end{array}
\]
}

\bibliographystyle{maa}
\bibliography{999_refs,999_refs_rgk,999_refs_rgk_media}

\end{document}